\documentclass[oneside, 11pt]{amsart} 
\usepackage{amsmath,amsthm,amssymb,epic}
\usepackage{eqlist,eqparbox}
\usepackage{amsfonts}
\usepackage{latexsym}
\usepackage[2emode]{psfrag}
\usepackage{amsthm}
\usepackage{amsmath}
\usepackage[all]{xy}

\newcommand{\Title}[1]{\bigskip\bigskip\centerline{\bf #1}\bigskip}
\newcommand{\Author}[1]{\medskip\centerline{ \it #1}}

\newcommand{\Affiliation}[1]{\medskip\centerline{#1}}
\newcommand{\Email}[1]{\medskip\centerline{#1}\bigskip}

\begin{document}

\newcommand{\N}{\mbox {$\mathbb N $}}
\newcommand{\Z}{\mbox {$\mathbb Z $}}
\newcommand{\Q}{\mbox {$\mathbb Q $}}
\newcommand{\R}{\mbox {$\mathbb R $}}
\newcommand{\lo }{\longrightarrow }
\newcommand{\ul}{\underleftarrow }
\newcommand{\rl}{\underrightarrow }
\newcommand{\rs }{\rightsquigarrow }
\newcommand{\ra }{\rightarrow }
\newcommand{\dd }{\rightsquigarrow }
\newcommand{\ol }{\overline }
\newcommand{\la }{\langle }
\newcommand{\tr }{\triangle }
\newcommand{\xr }{\xrightarrow }
\newcommand{\de }{\delta }
\newcommand{\pa }{\partial }
\newcommand{\LR }{\Longleftrightarrow }
\newcommand{\Ri }{\Rightarrow }
\newcommand{\va }{\varphi }
\newcommand{\Den}{{\rm Den}\,}
\newcommand{\Ker}{{\rm Ker}\,}
\newcommand{\Reg}{{\rm Reg}\,}
\newcommand{\Fix}{{\rm Fix}\,}
\newcommand{\Sup}{{\rm Sup}\,}
\newcommand{\Inf}{{\rm Inf}\,}
\newcommand{\Img}{{\rm Im}\,}
\newcommand{\Id}{{\rm Id}\,}

\newtheorem{theorem}{Theorem}[section]
\newtheorem{lemma}[theorem]{Lemma}
\newtheorem{proposition}[theorem]{Proposition}
\newtheorem{corollary}[theorem]{Corollary}
\newtheorem{definition}[theorem]{Definition}
\newtheorem{example}[theorem]{Example}
\newtheorem{examples}[theorem]{Examples}
\newtheorem{xca}[theorem]{Exercise}
\theoremstyle{remark}
\newtheorem{remark}[theorem]{Remark}
\numberwithin{equation}{section}

\def\leftmark{L.C. Ciungu}
%\title{Monadic operators on classes of quantum B-algebras}
%\author{Lavinia Corina Ciungu\\
%\footnotesize e-mail:lavinia-ciungu@uiowa.edu}

\Title{QUANTUM B-ALGEBRAS WITH INVOLUTIONS} 
\title[Quatum B-algebras with involutions]{}
                                                                           
\Author{\textbf{LAVINIA CORINA CIUNGU}}
\Affiliation{Department of Mathematics} 
\Affiliation{University of Iowa}
\Affiliation{14 MacLean Hall, Iowa City, Iowa 52242-1419, USA}
\Email{lavinia-ciungu@uiowa.edu}

\begin{abstract} 
The aim of this paper is to define and study the involutive and weakly involutive quantum B-algebras. 
We prove that any weakly involutive quantum B-algebra is a quantum B-algebra with pseudo-product.  
As an application, we introduce and investigate the notions of existential and universal quantifiers on involutive 
quantum B-algebras. It is proved that there is a one-to-one correspondence between the quantifiers on weakly 
involutive quantum B-algebras. One of the main results consists of proving that any pair of quantifiers is a 
monadic operator on weakly involutive quantum B-algebras. We investigate the relationship between quantifiers on 
bounded sup-commutative pseudo BCK-algebras and quantifiers on other related algebraic structures, such as pseudo MV-algebras and bounded Wajsberg hoops. \\

\textbf{AMS classification (2010):} 03G25, 06F35, 003B52
\end{abstract}

\maketitle

\section{Introduction} 

In the last decades, developing algebraic models for non-commutative multiple-valued logics became a central 
topic in the study of fuzzy systems.  
The non-commutative generalizations of MV-algebras called pseudo MV-algebras were introduced by G. Georgescu 
and A. Iorgulescu (\cite{Geo2}) and independently by J. Rach{\accent23u}nek (\cite{Rac2}) under the name of 
generalized MV-algebras. 
Pseudo-effect algebras were defined and investigated in \cite{DvVe1} and \cite{DvVe2} by A. Dvure\v censkij and 
T. Vetterlein as non-commutative generalizations of effect algebras. 
Pseudo BL-algebras were introduced and studied in \cite{Dino2} and \cite{Dino3}, pseudo MTL-algebras in \cite{Flo2}, 
bounded non-commutative R$\ell$-monoids in \cite{DvRa1} and pseudo-hoops in \cite{Geo16}. 
Pseudo BCK-algebras were introduced by G. Georgescu and A. Iorgulescu in \cite{Geo15} as algebras 
with "two differences", a left- and right-difference, and with a constant element $0$ as the least element. 
Nowadays, pseudo BCK-algebras are used in a dual form, with two implications, $\ra$ and $\rs$ and 
with one constant element $1$, that is the greatest element. 
Pseudo-BCI algebras were defined by \cite{Dud1} as generalizations of pseudo-BCK algebras and BCI-algebras, 
and they form an important tool for an algebraic axiomatization of implicational fragment of non-classical 
logic (\cite{Dym3}). \\
Rump and Yang introduced the concept of quantum B-algebras (\cite{Rump2,Rump1}), and proved that the quantum 
B-algebras can provide a unified semantic for non-commutative algebraic logic. All implicational algebras studied before - pseudo-effect algebras, residuated lattices, pseudo MV/BL/MTL-algebras, bounded non-commutative R$\ell$-monoids, pseudo-hoops, pseudo BCK/BCI-algebras - are quantum B-algebras. 
The properties of quantum B-algebras were investigated in \cite{Conig1}, \cite{Han1}, \cite{Han2}, \cite{Rump3}. \\
The concept of a \emph{existential quantifier} on a Boolean algebra $A$ was introduced by Halmos in \cite{Halmos} 
as a map $\exists:A\longrightarrow A$, such that: 
$(i)$ $\exists \perp=\perp$ ($\perp$ is the smallest element of $A$), $(ii)$ $x\le \exists x$, 
$(iii)$ $\exists(x\wedge \exists y)=\exists x\wedge \exists y$, for all $x, y\in A$. 
The pair $(A,\exists)$ was called \emph{Boolean monadic algebra} and the theory of monadic Boolean algebras is an algebraic treatment of the logic of propositional functions of one argument, with Boolean operations and a single
(existential) quantifier. 
The properties of the new concept have been studied by many researchers. 
Algebraic counterparts of the existential or universal quantifiers have been consequently 
studied also for other non-classical logics, while Halmos's representation theorems have been extended to certain 
algebras of fuzzy logic. 
Various methods were used to introduce the quantifiers on these algebraic structures (see \cite{Sol1}): \\
$\hspace*{0.5cm}$ $-$ to introduce pairs of existential and universal quantifiers and study their properties 
with respect to each other (\cite{Bez2}, \cite{Cas1}, \cite{Rac3}, \cite{Ciu35}, \cite{Wang1}, \cite{Zah1}, \cite{Ciu36}); \\
$\hspace*{0.5cm}$ $-$ to derive one quantifier from the other using a suitable involution-like operation 
(\cite{Dino1}, \cite{Geo21}, \cite{Ior16}, \cite{Sol1}, \cite{Cha1}). \\
Monadic MV-algebras were introduced and studied in \cite{Rut1} as an algebraic model of the predicate calculus 
of the  \L ukasiewicz infinite valued logic in which only a single individual variable occurs. 
Recently, the theory of monadic MV-algebras has been developed in \cite{Bell1}, \cite{Dino1}, \cite{Geo20}.
Monadic operators were defined and investigated on various algebras of fuzzy logic: 
Heyting algebras (\cite{Bez1}), basic algebras (\cite{Cha1}), GMV-algebras (\cite{Rac1}), involutive pseudo 
BCK-algebras (\cite{Ior16}), bounded commutative R$\ell$-monoids (\cite{Rac4}), bounded residuated lattices (\cite{Rac3}), residuated lattices (\cite{Kondo1}), 
BE-algebras (\cite{Zah1}), Wajsberg hoops (\cite{Cim1}), BL-algebras (\cite{Cas1}), bounded hoops (\cite{Wang1}), 
pseudo equality algebras (\cite{Ghor1}), pseudo BCI-algebras (\cite{Xin1}), NM-algebras (\cite{Wang2}), 
pseudo BE-algebras (\cite{Ciu35}). The monadic operators on quantum B-algebras have been recently introduced 
in \cite{Ciu36}. \\
The aim of this paper is to define and study the involutive and weakly involutive quantum B-algebras. 
We prove that any weakly involutive quantum B-algebra is a quantum B-algebra with pseudo-product. 
The notions of good and synchronized maps on a weakly involutive quantum B-algebra $X$ are defined and, given a 
pair $(\tau,\sigma)$ of synchronized maps, we prove that $\tau$ is an interior operator if and only if $\sigma$ 
is a closure operator on $X$. 
As an application, we introduce and investigate the notions of existential and universal quantifiers on involutive 
quantum B-algebras. It is proved that there is a one-to-one correspondence between the quantifiers on weakly 
involutive quantum B-algebras. One of the main results consists of proving that any pair of quantifiers is a 
monadic operator on weakly involutive quantum B-algebras. The case of weakly involutive integral quantum B-algebras 
is also considered. Finally, we investigate the relationship between quantifiers on bounded sup-commutative pseudo BCK-algebras and quantifiers on other related algebraic structures, such as pseudo MV-algebras and bounded Wajsberg hoops.

$\vspace*{5mm}$

\section{Preliminaries}

In this section we recall some basic notions and results regarding quantum B-algebras used in the paper.  

\begin{definition} \label{psBE-10} $\rm($\cite{Rump2}$\rm)$ 
A \emph{quantum B-algebra} is a partially ordered set $(X,\le)$ with two binary operations $\ra$ and $\rs$ 
satisfying the following axioms, for all $x, y, z\in X:$ \\
$(QB_1)$ $y\ra z\le (x\ra y)\ra (x\ra z);$ \\
$(QB_2)$ $y\rs z\le (x\rs y)\ra (x\rs z);$ \\
$(QB_3)$ $y\le z$ implies $x\ra y\le x\ra z;$ \\
$(QB_4)$ $x\le y\ra z$ iff $y\le x\rs z$.   
\end{definition}

We will refer to $(X,\le,\ra,\rs)$ by its univers $X$. 
A quantum B-algebra $(X,\le,\ra,\rs)$ is said to be \emph{commutative} if $x\ra y=x\rs y$, for all $x, y\in X$. 

\begin{proposition} \label{psBE-20} $\rm($\cite{Rump1, Rump2, Han1}$\rm)$ 
Let $(X,\le,\ra,\rs)$ be a quantum B-algebra. The following hold, for all $x, y, z\in X:$ \\
$(1)$ $x\le (x\ra y)\rs y$, $x\le (x\rs y)\ra y;$ \\
$(2)$ $y\le z$ implies $x\rs y\le x\rs z;$ \\
$(3)$ $x\le y$ implies $y\ra z\le x\ra z$, $y\rs z\le x\rs z;$ \\ 
$(4)$ $x\ra y=((x\ra y)\rs y)\ra y$, $x\rs y=((x\rs y)\ra y)\rs y;$ \\
$(5)$ $x\ra y\le (y\ra z)\rs (x\ra z)$, $x\rs y\le (y\rs z)\ra (x\rs z);$ \\
$(6)$ $x\ra(y\rs z)=y\rs (x\ra z)$.  
\end{proposition}
\begin{proof}
$(3)$ From $x\le y\le (y\ra z)\rs z$, it follows that $y\ra z\le x\ra z$. Similarly, $y\rs z\le x\rs z$. \\
$(5)$ It follows by $(QB_1)$ and $(QB_4)$. \\ 
$(6)$ From $(1)$ and $(QB_1)$ we have $y\le (y\rs z)\ra z\le (x\ra (y\rs z))\ra (x\ra z)$, hence, by $(QB_4)$, 
$x\ra(y\rs z)\le y\rs (x\ra z)$. Similarly, $y\rs (x\ra z)\le x\ra(y\rs z)$, so $x\ra(y\rs z)=y\rs (x\ra z)$. 
\end{proof}

\begin{proposition} \label{psBE-30} $\rm($\cite{Rump2}$\rm)$ 
An algebra $(X,\le,\ra,\rs)$ endowed with a partial order $\le$ and two binary operations $\ra$ and $\rs$ is a 
quantum B-algebra if and only if it satisfies $(QB_3)$, $(QB_4)$ and the identity $x\ra (y\rs z)=y\rs (x\ra z)$, 
for all $x, y, z\in X$. 
\end{proposition}

We recall some notions regarding the quantum B-algebras from \cite{Rump1} and \cite{Rump2}. \\ 
A quantum B-algebra $X$ is said to be \emph{unital} if there is an element $u\in X$ such that 
$u\ra x=u\rs x=x$, for all $x\in X$. The element $u$ is called a \emph{unit element} and the unit element  
is unique (\cite{Rump1}). 
We can easily check that $x\le y$ iff $u\le x\ra y$ iff $u\le x\rs y$. 
Indeed, $x\le y$ iff $x\le u\rs y$ iff $u\le x\ra y$, and $x\le y$ iff $x\le u\ra y$ iff $u\le x\rs y$. 
Since $x\le x$, it is clear that $u\le x\ra x$ and $u\le x\rs x$, for all $x\in X$. 
A unital quantum B-algebra $X$ is said to be \emph{normal} if $x\ra x=x\rs x=u$, for all $x\in X$ (\cite{Rump1}).   
Let $X$ be a unital quantum B-algebra. An element $x\in X$ is said to be \emph{integral} if $x\ra u=x\rs u=u$ 
(\cite{Rump2}).
The subset of all integral elements of $X$ is denoted by $I(X)$. 
If $I(X)=X$, we say that the unital quantum B-algebra $X$ is integral. 

\begin{proposition} \label{psBE-70} $\rm($\cite{Rump2}$\rm)$
For a quantum B-algebra $X$ the following are equivalent: \\
$(a)$ $X$ is a pseudo BCK-algebra;\\
$(b)$ $X$ is integral;\\
$(c)$ $X$ has a greatest element which is a unit element. 
\end{proposition}

\begin{remark} \label{psBE-70-10}
Any integral quantum B-algebra $(X,\le,\ra,\rs,u)$ is normal. \\
Indeed, for any $x\in X$, $x\le x$, so $u\le x\ra x$ and $u\le x\rs x$. Since $u$ is the greatest element of $X$, 
we get $x\ra x=x\rs x=u$. 
\end{remark}

A quantum B-algebra $X$ is called \emph{bounded} if $X$ admits a smallest element, denoted by $0$ (\cite{Rump2}).  

\begin{remark} \label{psBE-90}
$(1)$ A bounded quantum B-algebra has also a greatest element $1$. \\
Indeed, let $1=0\ra 0$.  
Since $0\le x\rs 0$, then $x\le 0\ra 0=1$, for all $x\in X$, hence $1$ is a greatest element of $X$. 
From $0\le x\ra 0$ we have $x\le 0\rs 0=1^{\prime}$, so $1^{\prime}$ is also a greatest element. 
But the greatest element is unique: $1\le 1^{\prime}$ and $1^{\prime}\le 1$ imply $1=1^{\prime}$. 
Hence $1=0\ra 0=0\rs 0$ is the greatest element of $X$. \\
$(2)$ $0\ra x=0\rs x=1$, for all $x\in X$. \\
Indeed, from $0\le x$ we get $0\ra 0\le 0\ra x$, so $1\le 0\ra x$, that is, $0\ra x=1$. 
Similarly, $0\rs x=1$. \\
$(3)$ $1\ra 1=1\rs 1=1$. \\
Indeed, we have $1\ra 1=1\ra (0\rs 1)=0\rs (1\ra 1)=1$, by $(2)$. Similarly, $1\rs 1=1$. \\
$(4)$ $x\le 1$ implies $x\ra 1=x\rs 1=1$. \\
Since $x\le 1$ implies $1\ra 1\le x\ra 1$, we have $1\le x\ra 1$. Hence, $x\ra 1=1$, and similarly, $x\rs 1=1$.  
\end{remark}

\begin{definition} \label{psBE-90-10} $\rm($\cite{Dud1}$\rm)$ A \emph{pseudo BCI-algebra} is a structure 
$(X,\le,\ra,\rs,1)$, where $\le$ is binary relation on $X$, $\ra$ and $\rs$ are binary operations on $X$ 
and $1$ is an element of $X$ satisfying the following axioms, for all $x, y, z \in A:$ \\
$(psBCI_1)$ $x \ra y \le (y \ra z) \rs (x \ra z)$, $x \rs y \le (y \rs z) \ra (x \rs z);$ \\
$(psBCI_2)$ $x\le (x\ra y)\rs y$, $x\le (x\rs y)\ra y;$ \\
$(psBCI_3)$ $x\le x;$ \\
$(psBCI_4)$ if $x\le y$ and $y\le x$, then $x=y;$ \\
$(psBCI_5)$ $x\le y$ iff $x\ra y =1$ iff $x\rs y=1$. 
\end{definition}

Every pseudo BCI-algebra satisfying $x\le 1$, for all $\in X$ is a \emph{pseudo BCK-algebra}. 
The pseudo BCK-algebras were intensively studied in \cite{Ior14}, \cite{Ciu2}, \cite{Kuhr6}. 
Pseudo BCI-algebras, and all their classes are normal quantum B-algebras.

$\vspace*{5mm}$

\section{Involutions on quantum B-algebras}

In this section we define the involutive and weakly involutive quantum B-algebras, and investigate their properties.  
We give a characterization of weakly involutive quantum B-algebras, and we prove that any weakly involutive quantum B-algebra is a quantum B-algebra with pseudo-product. 
Additionally, we show that in pointed integral quantum B-algebras $(X,\le)$, the weakly involutive element is
the bottom element of $(X,\le)$.  

\begin{definition} \label{inv-qb-10} Let $(X,\le,\ra,\rs)$ be a quantum B-algebra. An element $d\in X$ is called: \\
$(1)$ \emph{weakly involutive}, if $(x\ra d)\rs d=(x\rs d)\ra d=x$, for all $x\in X;$ \\
$(2)$ \emph{cyclic}, if $x\ra d=x\rs d$, for all $x\in X;$ \\
$(3)$ \emph{involutive}, if $d$ is both weakly involutive and cyclic. 
\end{definition}

\begin{definition} \label{inv-qb-20} The $(X,\le,\ra,\rs,d)$, where $(X,\le,\ra,\rs)$ is a quantum B-algebra and 
$d\in X$ is called a \emph{pointed} quantum B-algebra. A pointed quantum B-algebra is called: \\
$(1)$ \emph{weakly involutive}, if $d$ is weakly involutive; \\
$(2)$ \emph{involutive}, if $d$ is involutive. 
\end{definition}

If $X$ is a pointed quantum B-algebra, denote $x^{-}=x\ra d$ and $x^{\sim}=x\ra d$, for all $x\in X$. 
Obviously, if $d$ is cyclic, then $x^{-}=x^{\sim}$. If $X$ is weakly involutive, then $x^{-\sim}=x^{\sim-}=x$. 

\begin{remark} \label{inv-qb-20-10} Let $(X,\le,\ra,\rs,d,u)$ be a weakly involutive unital quantum B-algebra. Then: \\
$(1)$ $u^{-}=u^{\sim}=d$. \\
$(2)$ $d^{-}=d^{\sim}=u$. \\ 
Indeed, from $u^{-\sim}=u^{\sim-}=u$, we get $(u\ra d)\rs d=(u\rs d)\ra d=u$, so that $d\rs d=d\ra d=u$, that is, 
$d^{-}=d^{\sim}=u$. 
\end{remark}

\begin{proposition} \label{inv-qb-30} Let $(X,\le,\ra,\rs,d)$ be a pointed quantum B-algebra. 
The following hold, for all $x, y\in X:$ \\
$(1)$ $x\le x^{-\sim}$, $x\le x^{\sim-};$ \\
$(2)$ $x^{-\sim-}=x^{-}$, $x^{\sim-\sim}=x^{\sim};$ \\
$(3)$ $x\le y$ implies $y^{-}\le x^{-}$ and $y^{\sim}\le x^{\sim};$ \\
$(4)$ $x\ra y\le y^{-}\rs x^{-}$, $x\rs y\le y^{\sim}\ra x^{\sim};$ \\ 
$(5)$ $x\le y^{-}$ iff $y\le x^{\sim};$ \\
$(6)$ $x\ra y^{\sim}=y\rs x^{-}$, $x\rs y^{-}=y\ra x^{\sim};$ \\
$(7)$ $x\ra y^{-\sim}=y^{-}\rs x^{-}=x^{-\sim}\ra y^{-\sim}$,  
      $x\rs y^{\sim-}=y^{\sim}\ra x^{\sim}=x^{\sim-}\rs y^{\sim-};$ \\
$(8)$ $x\ra y^{\sim}=y^{\sim-}\rs x^{-}=x^{-\sim}\ra y^{\sim}$, 
      $x\rs y^{-}=y^{-\sim}\ra x^{\sim}=x^{\sim-}\rs y^{-};$ \\
$(9)$ if $X$ is unital, then $(x\ra y^{\sim-})^{\sim-}=x\ra y^{\sim-}$ and 
                             $(x\rs y^{-\sim})^{-\sim}=x\rs y^{-\sim}$. 
\end{proposition}
\begin{proof}
$(1)$,$(2)$ follow from Proposition \ref{psBE-20}$(1)$,$(4)$, for $y:=d$. \\
$(3)$,$(4)$ follow from Proposition \ref{psBE-20}$(3)$,$(5)$, for $z:=d$. \\
$(5)$ It follows by $(QB_4)$ for $z:=d$. \\
$(6)$ Applying $(4)$,$(1)$ and Proposition \ref{psBE-20}$(3)$, we get 
$x\ra y^{\sim}\le y^{\sim-}\rs x^{-}\le y\rs x^{-}$ and 
$x\rs y^{-}\le y^{-\sim}\ra x^{\sim}\le y\ra x^{\sim}$. 
In these inequalities we change $x$ and $y$ obtaining 
$y\ra x^{\sim}\le x\rs y^{-}$ and $y\rs x^{-}\le x\ra y^{\sim}$.
Hence, $x\ra y^{\sim}=y\rs x^{-}$ and $x\rs y^{-}=y\ra x^{\sim}$. \\
$(7)$ By $(6)$ we have $y\rs x^{-}=x\ra y^{\sim}$. 
Replacing $y$ with $y^{-}$ we get $y^{-}\rs x^{-}=x\ra y^{-\sim}$. 
Replacing $x$ by $x^{-\sim}$ in the last identity and using $(2)$ we have 
$y^{-}\rs x^{-}=x^{-\sim}\ra y^{-\sim}$. 
Thus, $x\ra y^{-\sim}=y^{-}\rs x^{-}=x^{-\sim}\ra y^{-\sim}$. 
Similarly, $x\rs y^{\sim-}=y^{\sim}\ra x^{\sim}=x^{\sim-}\rs y^{\sim-}$. \\
$(8)$ It follows by replacing in $(7)$ $y$ with $y^{\sim}$ and $y$ with $y^{-}$, respectively, and 
applying $(2)$. \\
$(9)$ Applying $(7)$ and Proposition \ref{psBE-20}$(6)$ we get \\
$u\le (x\ra y^{\sim-})\rs (x\ra y^{\sim-})=x\ra ((x\ra y^{\sim-})\rs y^{\sim-})=
x\ra ((x\ra y^{\sim-})^{\sim-}\rs y^{\sim-})=
(x\ra y^{\sim-})^{\sim-}\rs (x\ra y^{\sim-})$. 
Hence, $(x\ra y^{\sim-})^{\sim-}\le x\ra y^{\sim-}$. 
On the other hand, by $(1)$, $x\ra y^{\sim-}\le (x\ra y^{\sim-})^{\sim-}$. 
It follows that $(x\ra y^{\sim-})^{\sim-}=x\ra y^{\sim-}$.  
Similarly, $(x\rs y^{-\sim})^{-\sim}=x\rs y^{-\sim}$.
\end{proof}

The following two propositions are proved following an idea from \cite{Ciu31}. 

\begin{proposition} \label{inv-qb-40} Let $(X,\le,\ra,\rs,d)$ be a weakly involutive quantum B-algebra. 
The following hold, for all $x, y\in X:$ \\
$(1)$ $x\le y$ iff $y^{-} \le x^{-}$ iff $y^{\sim} \le x^{\sim};$ \\
$(2)$ $x^{-}\le y$ implies $y^{\sim}\le x$ and $x^{\sim}\le y$ implies $y^{-}\le x;$ \\
$(3)$ $x\ra y=y^{-}\rs x^{-}$, $x\rs y=y^{\sim}\ra x^{\sim};$ \\
$(4)$ $x^{\sim}\ra y=y^{-}\rs x$, $x^{-}\rs y=y^{\sim}\ra x;$ \\ 
$(5)$ $(x\ra y^{-})^{\sim}=(y\rs x^{\sim})^{-}$.  
\end{proposition}
\begin{proof}
$(1)$ It follows by Proposition \ref{inv-qb-30}$(3)$. \\
$(2)$ It follows by $(1)$. \\
$(3)$ It follows by Proposition \ref{inv-qb-30}$(7)$. \\
$(4)$ Applying Proposition \ref{inv-qb-30}$(7)$, we get  
$x^{\sim}\ra y=x^{\sim}\ra y^{-\sim}=y^{-}\rs x^{\sim-}=y^{-}\rs x$. Similarly, $x^{-}\rs y=y^{\sim}\ra x$. \\
$(5)$ For any $z\in X$, applying $(QB_4)$ and $(3)$ we have: \\
$\hspace*{2cm}$ $(x\ra y^{-})^{\sim}\le z$ iff $z^{-}\le x\ra y^{-}$ iff $x\le z^{-}\rs y^{-}=y\ra z$ \\
$\hspace*{4.7cm}$ iff $y\le x\rs z=z^{\sim}\ra x^{\sim}$ iff $z^{\sim}\le y\rs x^{\sim}$ \\ 
$\hspace*{4.7cm}$ iff $(y\rs x^{\sim})^{-}\le z$. \\
Hence, $(x\ra y^{-})^{\sim}=(y\rs x^{\sim})^{-}$. 
\end{proof}

\begin{proposition} \label{inv-qb-50} Let $(X,\le,\ra,\rs,d,u)$ be a pointed unital quantum B-algebra. 
The following are equivalent: \\
$(a)$ $X$ is weakly involutive;\\
$(b)$ $x\ra y=y^{-}\rs x^{-}$, $x\rs y=y^{\sim}\ra x^{\sim};$ \\
$(c)$ $x^{\sim}\ra y=y^{-}\rs x$, $x^{-}\rs y=y^{\sim}\ra x;$ \\
$(d)$ $x^{-}\le y$ implies $y^{\sim}\le x$ and $x^{\sim}\le y$ implies $y^{-}\le x$. 
\end{proposition}
\begin{proof}
$(a)\Rightarrow (b)$ follows by Proposition \ref{inv-qb-40}$(3)$. \\
$(b)\Rightarrow (c)$ is obvious. \\
$(c)\Rightarrow (d)$ From $x^{-}\le y$ we get $u\le x^{-}\rs y=y^{\sim}\ra x$, thus $y^{\sim}\le x$. 
Similarly, $x^{\sim}\le y$ implies $y^{-}\le x$. \\
$(d)\Rightarrow (a)$ Since $x^{-}\le x^{-}$, we get $x^{-\sim}\le x$, so that $x^{-\sim}=x$. 
Similarly, $x^{\sim-}=x$, hence $X$ is weakly involutive.  
\end{proof}

\begin{proposition} \label{inv-qb-60} Let $(X,\le,\ra,\rs,d,u)$ be a weakly involutive integral quantum B-algebra. 
Then $d$ is the bottom element of $(X,\le)$. 
\end{proposition}
\begin{proof}
Since $X$ is integral, $u$ is the top element. For any $x\in X$ we have $x\le u$, so $u^{-}\le x^{-}$, 
that is, $u\ra d\le x^{-}$. Thus, $d\le x^{-}$. 
On the other hand, since $X$ is weakly involutive, it is easy to see that 
$\{x^{-}\mid x\in X\}=\{x^{\sim}\mid x\in X\}=X=\{x\in X\mid x^{-\sim}=x^{\sim-}=x\}$ 
(Obviously, $\{x^{-}\mid x\in X\}, \{x^{\sim}\mid x\in X\}\subseteq X$. Conversely, if $x\in X$, then 
$x=(x^{\sim})^{-}\in \{x^{-}\mid x\in X\}$ and $x=(x^{-})^{\sim}\in \{x^{\sim}\mid x\in X\}$, so that 
$X\subseteq \{x^{-}\mid x\in X\}, \{x^{\sim}\mid x\in X\}$). 
Hence, for any $y\in X$, there exists $x\in X$ such that $y=x^{-}$, and so $d\le y$. 
It follows that $d$ is the bottom element of $(X\le)$. 
\end{proof}

\begin{definition} \label{inv-qb-70} $\rm($\cite{Ciu36}$\rm)$ A \emph{quantum B-algebra with pseudo-product condition} or a quantum B(pP)-algebra for short, is a quantum B-algebra $(X,\le,\ra,\rs)$ satisfying the condition \\
(pP) for all $x,y\in X$, $x\odot y$ exists, where \\
$\hspace*{2cm}$ $x\odot y=\min\{z\in X \mid x\le y\ra z\}=\min\{z\in X \mid y\le x\rs z\}$. 
\end{definition}

Clearly, if $x\odot y$ exists, then it is unique and it satisfies $x\odot y\le z$ iff $x\le y\ra z$ iff $y\le x\rs z$. 
A quantum B(pP)-algebra is denoted by $(X,\le,\ra,\rs,\odot)$. 

\begin{proposition} \label{inv-qb-80} $\rm($\cite{Ciu36}$\rm)$ Let $(X,\le,\ra,\rs,\odot)$ be a quantum B(pP)-algebra. Then the following hold, for all $x, y, z\in X:$ \\
$(1)$ $\odot$ is associative; \\
$(2)$ $(x\ra y)\odot x\le y$, $x\odot (x\rs y)\le y;$ \\
$(3)$ $x\ra (y\ra z)=x\odot y\ra z$, $x\rs (y\rs z)=y\odot x\rs z;$ \\
$(4)$ $x\le y$ implies $x\odot z\le y\odot z$, $z\odot x\le z\odot y;$ \\
$(5)$ $(y\ra z)\odot (x\ra y)\le x\ra z$, $(x\rs y)\odot (y\rs z)\le x\rs z;$ \\
$(6)$ $x\ra y\le x\odot z\ra y\odot z$, $x\rs y\le z\odot x\rs z\odot y$. 
\end{proposition}

\begin{theorem} \label{inv-qb-90} Any weakly involutive quantum B-algebra is a quantum B(pP)-algebra. 
\end{theorem}
\begin{proof}
Let $(X,\le,\ra,\rs,d)$ be a weakly involutive quantum B-algebra.
Define the operation $\odot$ by $x\odot y=(x\ra y^{-})^{\sim}=(y\rs x^{\sim})^{-}$, for all $x,y\in X$. \\ 
First of all, for all $x,y\in X$, we have $x\odot y\le z$ iff $x\le y\ra z$ iff $y\le x\rs z$. \\
Indeed, $x\odot y\le z$ iff $(x\ra y^{-})^{\sim}\le z$ iff $z^{-}\le x\ra y^{-}$ 
                        iff $x\le z^{-}\rs y^{-}=y\ra z$. \\
Similarly, $x\odot y\le z$ iff $(y\rs x^{\sim})^{-}\le z$ iff $z^{\sim}\le y\rs x^{\sim}$ 
                           iff $y\le z^{\sim}\ra x^{\sim}=x\rs z$. \\
Now, we show that $x\odot y=\min \{z\mid x\le y\ra z\}$. 
Indeed, from $x\odot y\le x\odot y$ we get $x\le y\ra x\odot y$. If $z\in X$ verifies $x\le y\ra z$, then 
$x\odot y\le z$. Hence, $x\odot y=\min \{z\mid x\le y\ra z\}$. 
Similarly, $x\odot y=\min \{z\mid y\le x\rs z\}$.                     
\end{proof}

\begin{proposition} \label{inv-qb-100} Let $(X,\le,\ra,\rs,\odot,d)$ be a weakly involutive quantum B-algebra. 
The following hold, for all $x,y\in X:$ \\
$(1)$ $x\le y^{-}$ iff $x\odot y\le d$  and $x\le y^{\sim}$ iff $y\odot x\le d;$ \\ 
$(2)$ $y^{-}\odot (x\ra y)\le x^{-}$, $(x\rs y)\odot y^{\sim}\le x^{\sim};$ \\
$(3)$ $x^{-}\odot x\le d$, $x\odot x^{\sim}\le d;$ \\
$(4)$ $(x\odot y)^{-}=x\ra y^{-}$, $(y\odot x)^{\sim}=x\rs y^{\sim};$ \\
$(5)$ $(x^{-}\odot y^{-})^{\sim}=(x^{\sim}\odot y^{\sim})^{-};$ \\
$(6)$ $x\odot y=(x\ra y^{-})^{\sim}=(y\rs x^{\sim})^{-};$ \\ 
$(7)$ $x\ra y=(x\odot y^{\sim})^{-}$, $x\rs y=(y^{-}\odot x)^{\sim}$. 
\end{proposition}
\begin{proof}
$(1)$ It is straightforward. \\
$(2)$ It follows by Proposition \ref{inv-qb-80}$(5)$ for $z:=d$. \\
$(3)$ From $x\ra d\le x\ra d$ we have $(x\ra d)\odot x\le d$, so that $x^{-}\odot x\le d$. 
Similarly, $x\rs d\le x\rs d$ implies $x\odot (x\rs d)\le d$, thus $x\odot x^{\sim}\le d$. \\
$(4)$ Applying Proposition \ref{inv-qb-80}$(3)$, we get 
$(x\odot y)^{-}=x\odot y\ra d=x\ra (y\ra d)=x\ra y^{-}$. 
Similarly, $y\odot x)^{\sim}=y\odot x\rs d=x\rs (y\rs d)=x\rs y^{\sim}$. \\
$(5)$ Aplying $(4)$ and Proposition \ref{inv-qb-40}$(4)$, we get: 
$(x\odot y)^{-}=y^{-}\rs x^{-\sim}=y^{-}\rs x=x^{\sim}\ra y=x^{\sim}\ra y^{\sim-}=(x^{\sim}\odot y^{\sim})^{-}$. \\
$(6)$,$(7)$ follow from $(4)$. 
\end{proof}

Let $(X,\le,\ra,\rs,\odot,d)$ be a weakly involutive quantum B-algebra. Based on Proposition \ref{inv-qb-50}, 
we define the operation $\oplus$ by: \\
%$\hspace*{2cm}$ 
$(S)$ $x\oplus y=y^{\sim}\ra x=x^{-}\rs y$, for all $x,y \in X$. 

\begin{lemma} \label{inv-qb-100-10} Let $(X,\le,\ra,\rs,\odot,d)$ be a weakly involutive quantum B-algebra. 
The following hold, for all $x,y\in X:$ \\
$(1)$ $x\odot y=(y^{-}\oplus x^{-})^{\sim}=(y^{\sim}\oplus x^{\sim})^{-};$ \\
$(1)$ $x\oplus y=(y^{-}\odot x^{-})^{\sim}=(y^{\sim}\odot x^{\sim})^{-}$. 
\end{lemma}
\begin{proof} The proof is straightforward. 
\end{proof}

\begin{example} \label{inv-qb-110} $\rm($\cite{Will1}$\rm)$
Let $X=\{0,a,b,c,u,1\}$ be a poset with $\le$ defined by $0\le a\le b\le u\le 1$ and $a\le c\le u$ 
(see the diagram below). 

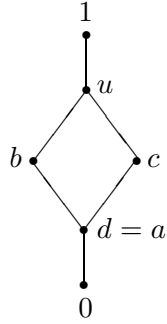
\begin{center}
\begin{picture}(50,155)(0,0)

\put(37,40){\circle*{3}}
\put(35,28){$0$}
\put(37,40){\line(0,1){20}}

\put(37,61){\circle*{3}}
\put(42,58){$d=a$}
\put(37,61){\line(3,4){20}}
\put(57,87){\circle*{3}}
\put(61,85){$c$}

\put(37,61){\line(-3,4){20}}
\put(18,87){\circle*{3}}
\put(9,85){$b$}

\put(18,87){\line(3,4){20}}
\put(38,114){\circle*{3}}
\put(42,113){$u$}

\put(38,114){\line(3,-4){20}}

\put(38,114){\line(0,1){20}}
\put(38,135){\circle*{3}}
\put(35,140){$1$}

\end{picture}
\end{center}

Define the operations $\ra$ and $\rs$ on $X$ by the following tables: 
\[
\begin{array}{c|cccccc}
\ra & 0 & a & b & c & u & 1 \\ \hline
0 & 1 & 1 & 1 & 1 & 1 & 1 \\
a & c & u & u & 1 & 1 & 1 \\
b & c & c & u & c & u & 1 \\
c & 0 & b & b & 1 & 1 & 1 \\
u & 0 & a & b & c & u & 1 \\
1 & 0 & 0 & b & 0 & b & 1 
\end{array}
\hspace{10mm}
\begin{array}{c|cccccc}
\rs & 0 & a & b & c & u & 1 \\ \hline
0 & 1 & 1 & 1 & 1 & 1 & 1 \\
a & b & u & 1 & u & 1 & 1 \\
b & 0 & c & 1 & c & 1 & 1 \\
c & b & b & b & u & u & 1 \\
u & 0 & a & b & c & u & 1 \\
1 & 0 & 0 & 0 & c & c & 1 
\end{array}
.
\]

Then $(X,\le,\ra,\rs,d=a,u)$ is a bounded involutive unital quantum B-algebra, with $d=a$. 
Moreover, $X$ is a quantum B(pP)-algebra with the operation $\odot$ defined below.  

\[
\begin{array}{c|cccccc}
\odot & 0 & a & b & c & u & 1 \\ \hline
0 & 0 & 0 & 0 & 0 & 0 & 0 \\
a & 0 & 0 & 0 & a & a & b \\
b & 0 & a & b & a & b & b \\
c & 0 & 0 & 0 & c & c & 1 \\
u & 0 & a & b & c & u & 1 \\
1 & 0 & c & 1 & c & 1 & 1 
\end{array}
.
\]

\end{example}

$\vspace*{5mm}$
\section{Good maps on weakly involutive quantum B-algebras} 

The notions of good, synchronized and strong synchronized maps on a weakly involutive quantum B-algebra $X$ are defined and investigated. Given a pair $(\tau,\sigma)$ of synchronized maps, we prove that $\tau$ is an interior operator if and only if $\sigma$ is a closure operator on $X$. We also show that for any pair $(\tau,\sigma)$ of synchronized maps, the pair $(\tau\sigma,\sigma\tau)$ is also a synchronized map. 
Recall that a closure operator on a poset $(X,\le)$ is a map $\gamma$ that is, increasing, isotone and idempotent, i.e. 
$x\le \gamma(x)$, $x\le y$ implies $\gamma(x)\le \gamma(y)$ and $\gamma(\gamma(x))=\gamma(x)$, for all $x\in A$. 
Dually, an interior operator is a decreasing ($\gamma(x)\le x$), isotone and idempotent map on $(X,\le)$ 
(see \cite{Gal3}). 

\begin{definition} \label{gm-10} Let $(X,\le,\ra,\rs,d)$ be a weakly involutive quantum B-algebra. A map $\tau:X\longrightarrow X$ is said to be \emph{good} if $(\tau x^{-})^{\sim}=(\tau x^{\sim})^{-}$, for all $x\in X$. 
\end{definition} 

Denote by $\mathcal{GM}(X)$ the set of all good maps on $X$. We use the notation $\tau x$ instead of $\tau(x)$. 
If $\tau_1,\tau_2:X\longrightarrow X$, then the composition $\tau_1\circ \tau_2$ will be denoted by $\tau_1\tau_2$.    

\begin{remark} \label{gm-20} Let $(X,\le,\ra,\rs,d)$ be a weakly involutive quantum B-algebra. \\
$(1)$ $Id_X\in \mathcal{GM}(X)$. \\
$(2)$ If $X$ is cyclic, then any map on $X$ is good. \\
$(3)$ If $X$ is commutative, then any map on $X$ is good. 
\end{remark}

\begin{lemma} \label{gm-30} Let $(X,\le,\ra,\rs,d)$ be a weakly involutive quantum B-algebra and let 
$\tau\in \mathcal{GM}(X)$. Define $\sigma:X\longrightarrow X$, by 
$\sigma x=(\tau x^{-})^{\sim}=(\tau x^{\sim})^{-}$, for all $x\in X$. 
Then $\sigma\in \mathcal{GM}(X)$ and $\tau x=(\sigma x^{-})^{\sim}=(\sigma x^{\sim})^{-}$. 
\end{lemma}
\begin{proof}
Replacing $x$ with $x^{\sim}$ in $\sigma x=(\tau x^{-})^{\sim}$, we get $\sigma x^{\sim}=(\tau x)^{\sim}$, that 
is, $(\sigma x^{\sim})^{-}=\tau x$. Similarly, from $\sigma x=(\tau x^{\sim})^{-}$, replacing $x$ with $x^{-}$, we 
get $(\sigma x^{-})^{\sim}=\tau x$. It follows that $\tau x=(\sigma x^{-})^{\sim}=(\sigma x^{\sim})^{-}$, hence 
$\sigma$ is a good map on $X$. 
\end{proof}

\begin{definition} \label{gm-40} The good maps $\tau$ and $\sigma$ defined in Lemma \ref{gm-30} are called 
\emph{synchronized}. 
\end{definition}

Denote by $\mathcal{SM}(X)$ the set of all pairs of synchronized maps on $X$. 
Obviously, if $(\tau,\sigma)\in \mathcal{SM}(X)$, then $(\sigma,\tau)\in \mathcal{SM}(X)$. 

\begin{proposition} \label{gm-50} Let $(X,\le,\ra,\rs,d)$ be a weakly involutive quantum B-algebra and let $(\tau,\sigma) \in \mathcal{SM}(X)$. Then the following hold, for all $x\in X:$ \\
$(1)$ $(\tau x)^{-}=\sigma x^{-}$, $(\tau x)^{\sim}=\sigma x^{\sim};$ \\
$(2)$ $(\sigma x)^{-}=\tau x^{-}$, $(\sigma x)^{\sim}=\tau x^{\sim};$ \\ 
$(3)$ $\sigma\tau x=(\tau\sigma x^{-})^{\sim}=(\tau\sigma x^{\sim})^{-};$ \\
$(4)$ $\tau\sigma x=(\sigma\tau x^{-})^{\sim}=(\sigma\tau x^{\sim})^{-}$.  
\end{proposition}
\begin{proof}
$(1)$ It follows replacing $x$ with $x^{-}$ and $x$ with $x^{\sim}$ in $\sigma x=(\tau x^{\sim})^{-}$ and 
$\sigma x=(\tau x^{-})^{\sim}$, respectively. \\
$(2)$ It follows replacing $x$ with $x^{-}$ and $x$ with $x^{\sim}$ in $\tau x=(\sigma x^{\sim})^{-}$ and 
$\tau x=(\sigma x^{-})^{\sim}$, respectively. \\ 
$(3)$ Using $(2)$, we get $\tau\sigma x=(\sigma(\sigma x)^{-})^{\sim}=(\sigma\tau x^{-})^{\sim}$, hence 
$(\tau\sigma x)^{-}=\sigma\tau x^{-}$. Replacing $x$ with $x^{\sim}$ we get $\sigma\tau x=(\tau\sigma x^{\sim})^{-}$. 
Similarly, $\tau\sigma x=(\sigma(\sigma x)^{\sim})^{-}=(\sigma\tau x^{\sim})^{-}$, so that 
$(\tau\sigma x)^{\sim}=\sigma\tau x^{\sim}$. Replacing $x$ with $x^{-}$, we have 
$\sigma\tau x=(\tau\sigma x^{-})^{\sim}$. 
It follows that $\sigma\tau x=(\tau\sigma x^{-})^{\sim}=(\tau\sigma x^{\sim})^{-}$. \\
$(4)$ Similarly to $(3)$, applying $(1)$. 
\end{proof}

\begin{corollary} \label{gm-50-10} Let $(X,\le,\ra,\rs,d)$ be a weakly involutive quantum B-algebra.  
If $(\tau,\sigma)\in \mathcal{SM}(X)$, then $(\tau\sigma,\sigma\tau)\in \mathcal{SM}(X)$. 
\end{corollary}

\begin{proposition} \label{gm-60} Let $(X,\le,\ra,\rs,d)$ be a weakly involutive quantum B-algebra and let $(\tau,\sigma) \in \mathcal{SM}(X)$. Then the following hold: \\
$(1)$ $\tau\sigma=\sigma$ iff $\sigma\tau=\tau;$ \\
$(2)$ $\tau\tau=\tau$ iff $\sigma\sigma=\sigma$. 
\end{proposition}
\begin{proof}
$(1)$ If $\tau\sigma=\sigma$, then $\tau\sigma x^{-}=\sigma x^{-}$, for all $x\in X$. 
By Proposition \ref{gm-50}$(1)$, $\tau(\tau x)^{-}=(\tau x)^{-}$, so that $(\tau(\tau x)^{-})^{\sim}=\tau x$. 
Hence $\sigma\tau x=\tau x$, for all $x\in X$, that is, $\sigma\tau=\tau$. 
Similarly, $\sigma\tau=\tau$ implies $\tau\sigma=\sigma$. \\
$(2)$ Assume that $\tau\tau=\tau$, so that $\tau\tau x^{-}=\tau x^{-}$, for all $x\in X$. 
Applying Proposition \ref{gm-50}$(2)$ we get $\tau(\sigma x)^{-}=(\sigma x)^{-}$, thus 
$(\tau(\sigma x)^{-})^{\sim}=\sigma x$, for all $x\in X$, that is $\sigma\sigma x=\sigma x$, for all $x\in X$. 
It follows that $\sigma\sigma=\sigma$. Similarly, $\sigma\sigma=\sigma$ implies $\tau\tau=\tau$. 
\end{proof}

\begin{corollary} \label{gm-60-20} Let $(X,\le,\ra,\rs,d)$ be a weakly involutive quantum B-algebra and let   
$(\tau,\sigma)\in \mathcal{SM}(X)$. Then $\tau$ is idempotent if and only if $\sigma$ is idempotent.  
\end{corollary}

\begin{definition} \label{gm-60-30} Let $(X,\le,\ra,\rs,d)$ be a weakly involutive quantum B-algebra and let 
$(\tau,\sigma)\in \mathcal{SM}(X)$ satisfying one of the equivalent conditions from Proposition \ref{gm-60}$(1)$. 
Then the maps $\tau$ and $\sigma$ are called \emph{strong synchronized}. 
\end{definition}

Denote by $\mathcal{SSM}(X)$ the set of all pairs of strong synchronized maps on $X$. 

\begin{proposition} \label{gm-80} Let $(X,\le,\ra,\rs,d)$ be a weakly involutive quantum B-algebra and let $(\tau,\sigma) \in \mathcal{SM}(X)$. Then the following hold: \\
$(1)$ $\tau$ is decreasing if and only if $\sigma$ is increasing; \\
$(2)$ $\tau$ is isotone (antitone) if and only if $\sigma$ is isotone (antitone). 
\end{proposition}
\begin{proof}
$(1)$ Assume that $\tau$ is decreasing. Hence, for any $x\in X$, we have $\tau x^{-}\le x^{-}$, so that 
$x\le (\tau x^{-})^{\sim}$. It follows that $x\le \sigma x$, that is $\sigma$ is increasing. 
Conversely, if $\sigma$ is increasing, then $x^{-}\le \sigma x^{-}$, for all $x\in X$. 
We get $(\sigma x^{-})^{\sim}\le x^{-\sim}=x$, that is, $\tau x\le x$. Thus, $\tau$ is decreasing. \\
$(2)$ Suppose that $\tau$ is isotone and let $x,y\in X$ such that $x\le y$, that is $y^{-}\le x^{-}$. 
Since $\tau$ is isotone, we get $\tau y^{-}\le \tau x^{-}$, so that $(\tau x^{-})^{\sim}\le (\tau y^{-})^{\sim}$, 
that is $\sigma x\le \sigma y$. We conclude that $\sigma$ is isotone. Similarly for the antitone case. 
\end{proof}

\begin{theorem} \label{gm-90} Let $(X,\le,\ra,\rs,d)$ be a weakly involutive quantum B-algebra and let   
$(\tau,\sigma)\in \mathcal{SM}(X)$. Then $\tau$ is an interior operator if and only if $\sigma$ is a closure 
operator on $X$.  
\end{theorem}
\begin{proof} 
It follows by Proposition \ref{gm-80} and Corollary \ref{gm-60-20}. 
\end{proof}

\begin{proposition} \label{gm-100} Let $(X,\le,\ra,\rs,d)$ be a weakly involutive quantum B-algebra and let $(\tau_1,\sigma_1), (\tau_2,\sigma_2) \in \mathcal{SM}(X)$. Then the following hold: \\
$(1)$ $\tau_1\le \tau_2$ iff $\sigma_1\ge \sigma_2;$ \\
$(2)$ $\sigma_1\le \sigma_2$ iff $\tau_1\ge \tau_2$. 
\end{proposition}
\begin{proof} The proof is straightforward. 
\end{proof}

\begin{proposition} \label{gm-110} Let $(X,\le,\ra,\rs,d,u)$ be a weakly involutive unital quantum B-algebra 
and let $(\tau,\sigma) \in \mathcal{SM}(X)$. Then the following hold: \\
$(1)$ $\tau d=d$ iff $\sigma u=u;$ \\
$(2)$ $\sigma d=d$ iff $\tau u=u$. 
\end{proposition}
\begin{proof}
$(1)$ Assume that $\tau d=d$. Using Remark \ref{inv-qb-20-10}, we have 
$\sigma u=(\tau u^{-})^{\sim}=(\tau(u\ra d))^{\sim}=(\tau d)^{\sim}=d^{\sim}=d\rs d=u$. 
Conversely, if $\sigma u=u$, then 
$\tau d=(\sigma d^{-})^{\sim}=(\sigma (d\ra d))^{\sim}=(\sigma u)^{\sim}=u^{\sim}=u\rs d=d$. \\
$(3)$ Similarly to $(2)$. 
\end{proof}

$\vspace*{5mm}$
\section{Quantifiers on weakly involutive quantum B-algebras}

We define the existential and universal quantifiers on weakly involutive quantum B-algebras, and we extend to 
the case of weakly involutive quantum B-algebras some results proved in \cite{Ior16} for involutive 
pseudo BCK-algebras. It is proved that there is a one-to-one correspondence between the quantifiers on weakly 
involutive quantum B-algebras. One of the main results consists of proving that any pair of quantifiers is a 
strict monadic operator on weakly involutive quantum B-algebras. 
%More precisely, we show that any pair of weakly quantifiers is a strict monadic operator. 
In this section, $(X,\le,\ra,\rs,d,u)$ will be a weakly involutive unital quantum B-algebra, unless otherwise stated.

\begin{definition} \label{qinv-10} A map $\exists\in \mathcal{GM}(X)$ is called an \emph{existential quantifier} 
on $X$ if it satisfies the following axioms, for all $x,y\in X:$ \\
$(E_1)$ $\exists d=d;$ \\
$(E_2)$ $\exists u=u;$ \\
$(E_3)$ $x\le \exists x;$ \\
$(E_4)$ $\exists(x\oplus \exists y)=\exists(\exists x\oplus y)=\exists x\oplus \exists y;$ \\ 
$(E_5)$ $\exists(x\oplus x)=\exists x\oplus \exists x;$ \\  
$(E_6)$ $\exists(x\odot \exists y)=\exists(\exists x\odot y)=\exists x\odot \exists y;$ \\ 
$(E_7)$ $\exists(x\odot x)=\exists x\odot \exists x$. 
\end{definition}

\begin{definition} \label{qinv-10-10} An existential quantifier $\exists$ on $X$ is said to be \emph{weak} if 
it satisfies the axioms $(E_1)-(E_6)$ from Definition \ref{qinv-10}.
\end{definition}

\begin{definition} \label{qinv-20} A map $\forall\in \mathcal{GM}(X)$ is called a \emph{universal quantifier} 
on $X$ if it satisfies the following axioms, for all $x,y\in X:$ \\
$(U_1)$ $\forall u=u;$ \\
$(U_2)$ $\forall d=d;$ \\
$(U_3)$ $\forall x\le x;$ \\
$(U_4)$ $\forall(x\odot \forall y)=\forall(\forall x\odot y)=\forall x\odot \forall y;$ \\ 
$(U_5)$ $\forall(x\odot x)=\forall x\odot \forall x;$ \\ 
$(U_6)$ $\forall(x\oplus \forall y)=\forall(\forall x\oplus y)=\forall x\oplus \forall y;$\\ 
$(U_7)$ $\forall(x\oplus x)=\forall x\oplus \forall x$.   
\end{definition} 

\begin{definition} \label{qinv-20-10} A universal quantifier $\forall$ on $X$ is said to be \emph{weak} if 
it satisfies the axioms $(U_1)-(U_6)$ from Definition \ref{qinv-20}.
\end{definition}

\begin{theorem} \label{qinv-30} The following hold:\\
$(1)$ if $(\exists,\forall)\in \mathcal{SM}(X)$ such that $\exists$ is an existential quantifier on $X$, 
then $\forall$ is a universal quantifier on $X;$ \\
$(2)$ if $(\exists,\forall)\in \mathcal{SM}(X)$ such that $\forall$ is a universal quantifier on $X$, 
then $\exists$ is an existential quantifier on $X$.  
\end{theorem}
\begin{proof}
Let $(\exists,\forall)\in \mathcal{SM}(X)$. We prove that $\exists$ is an existential quantifier if and only if 
$\forall$ is a universal quantifier on $X$, that is, the axioms $(E_1)-(E_7)$ are equivalent to axioms $(U_1)-(U_7)$. \\
$(E_1)\Leftrightarrow (U_1)$, $(E_2)\Leftrightarrow (U_2)$, $(E_3)\Leftrightarrow (U_3)$ follow from Propositions \ref{gm-110} and \ref{gm-80}$(1)$. \\
$(E_4)\Leftrightarrow (U_4)$ Suppose that $(E_4)$ holds and we have: \\
$\hspace*{2.00cm}$ $\forall(\forall x\odot y)=\forall(y^{-}\oplus (\forall x)^{-})^{\sim}$ (Lemma \ref{inv-qb-100-10})\\
$\hspace*{3.75cm}$ $=(\exists(y^{-}\oplus \exists x^{-}))^{\sim}$ (Prop. \ref{gm-50}) \\
$\hspace*{3.75cm}$ $=(\exists y^{-}\oplus \exists x^{-})^{\sim}$ (by $(E_4)$) \\
$\hspace*{3.75cm}$ $=((\forall y)^{-}\oplus (\forall x)^{-})^{\sim}$ (Prop. \ref{gm-50}) \\
$\hspace*{3.75cm}$ $=\forall x\odot \forall y$ (Lemma \ref{inv-qb-100-10}). \\
$\hspace*{2.00cm}$ $\forall(x\odot \forall y)=\forall((\forall y)^{-}\oplus x^{-})^{\sim}$ (Lemma \ref{inv-qb-100-10})\\
$\hspace*{3.75cm}$ $=(\exists(\exists y^{-}\oplus x^{-}))^{\sim}$ (Prop. \ref{gm-50}) \\
$\hspace*{3.75cm}$ $=(\exists y^{-}\oplus \exists x^{-})^{\sim}$ (by $(E_4)$) \\
$\hspace*{3.75cm}$ $=((\forall y)^{-}\oplus (\forall x)^{-})^{\sim}$ (Prop. \ref{gm-50}) \\
$\hspace*{3.75cm}$ $=\forall x\odot \forall y$ (Lemma \ref{inv-qb-100-10}). \\
Hence, $(E_4)$ implies $(U_4)$. Conversely, assume that $(U_4)$ holds, so we get: \\
$\hspace*{2.00cm}$ $\exists(\exists x\oplus y)=\exists(y^{-}\odot (\exists x)^{-})^{\sim}$ (Lemma \ref{inv-qb-100-10})\\
$\hspace*{3.75cm}$ $=(\forall(y^{-}\odot \forall x^{-}))^{\sim}$ (Prop. \ref{gm-50}) \\
$\hspace*{3.75cm}$ $=(\forall y^{-}\odot \forall x^{-})^{\sim}$ (by $(U_4)$) \\
$\hspace*{3.75cm}$ $=((\exists y)^{-}\odot (\exists x)^{-})^{\sim}$ (Prop. \ref{gm-50}) \\
$\hspace*{3.75cm}$ $=\exists x\oplus \exists y$ (Lemma \ref{inv-qb-100-10}). \\
$\hspace*{2.00cm}$ $\exists(x\oplus \exists y)=\exists((\exists y)^{-}\odot x^{-})^{\sim}$ (Lemma \ref{inv-qb-100-10})\\
$\hspace*{3.75cm}$ $=(\forall(\forall y^{-}\odot x^{-}))^{\sim}$ (Prop. \ref{gm-50}) \\
$\hspace*{3.75cm}$ $=(\forall y^{-}\odot \forall x^{-})^{\sim}$ (by $(U_4)$) \\
$\hspace*{3.75cm}$ $=((\exists y)^{-}\odot (\exists x)^{-})^{\sim}$ (Prop. \ref{gm-50}) \\
$\hspace*{3.75cm}$ $=\exists x\oplus \exists y$ (Lemma \ref{inv-qb-100-10}). \\
It follows that $(U_4)$ implies $(E_4)$, thus $(E_4)$ is equivalent to $(U_4)$. \\ 
$(E_5)\Leftrightarrow (U_5)$ Using $\exists(x\oplus x)=\exists x\oplus \exists x$, and Lemma \ref{inv-qb-100-10}, and Proposition \ref{gm-50}, we have: \\
$\hspace*{2.00cm}$ $\forall(x\odot x)=\forall(x^{-}\oplus x^{-})^{\sim}=(\exists(x^{-}\oplus x^{-}))^{\sim}=
                   (\exists x^{-}\oplus \exists x^{-})^{\sim}$ \\
$\hspace*{3.50cm}$ $=((\forall x)^{-}\oplus (\forall x)^{-})^{\sim}=\forall x\odot \forall x$. \\
Conversely, assuming that $(U_5)$ holds, we get: \\
$\hspace*{2.00cm}$ $\exists(x\oplus x)=\exists(x^{-}\odot x^{-})^{\sim}=(\forall(x^{-}\odot x^{-}))^{\sim}=
                   (\forall x^{-}\odot \forall x^{-})^{\sim}$ \\
$\hspace*{3.50cm}$ $=((\exists x)^{-}\odot (\exists x)^{-})^{\sim}=\exists x\oplus \exists x$. \\
Hence, $(E_5)$ is equivalent to $(U_5)$. \\
$(E_6)\Leftrightarrow (U_6)$ If $(E_6)$ holds, applying Lemma \ref{inv-qb-100-10} and Propositions \ref{gm-50}, 
we have: \\
$\hspace*{2.00cm}$ $\forall(x\oplus \forall y)=\forall((\forall y)^{-}\odot x^{-})^{\sim}=
                    (\exists(\exists y^{-}\odot x^{-}))^{\sim}=(\exists y^{-}\odot \exists x^{-})^{\sim}$ \\
$\hspace*{3.75cm}$ $=((\forall y)^{-}\odot (\forall x)^{-})^{\sim}=\forall x\oplus \forall y$. \\
$\hspace*{2.00cm}$ $\forall(\forall x\oplus y)=\forall(y^{-}\odot (\forall x)^{-})^{\sim}=
                    (\exists(y^{-}\odot \exists x^{-}))^{\sim}=(\exists y^{-}\odot \exists x^{-})^{\sim}$ \\
$\hspace*{3.75cm}$ $=((\forall y)^{-}\odot (\forall x)^{-})^{\sim}=\forall x\oplus \forall y$. \\
Conversely, assume that $(U_5)$ holds, so that: \\ 
$\hspace*{2.00cm}$ $\exists(x\odot \exists y)=\exists((\exists y)^{-}\oplus x^{-})^{\sim}=
                    (\forall(\forall y^{-}\oplus x^{-}))^{\sim}=(\forall y^{-}\oplus \forall x^{-})^{\sim}$ \\
$\hspace*{3.75cm}$ $=((\exists y)^{-}\oplus (\exists x)^{-})^{\sim}=\exists x\odot \exists y$. \\
$\hspace*{2.00cm}$ $\exists(\exists x\odot y)=\exists(y^{-}\oplus (\exists x)^{-})^{\sim}=
                    (\forall(y^{-}\oplus \forall x^{-}))^{\sim}=(\forall y^{-}\oplus \forall x^{-})^{\sim}$ \\
$\hspace*{3.75cm}$ $=((\exists y)^{-}\oplus (\exists x)^{-})^{\sim}=\exists x\odot \exists y$. \\
If follows that $(E_6)$ is equivalent to $(U_6)$. \\
$(E_7)\Leftrightarrow (U_7)$ Suppose that $(E_7)$ holds, and we have: \\
$\hspace*{2.00cm}$ $\forall(x\oplus x)=\forall(x^{-}\odot x^{-})^{\sim}=(\exists(x^{-}\odot x^{-}))^{\sim}=
                    (\exists x^{-}\odot \exists x^{-})^{\sim}$ \\
$\hspace*{3.50cm}$ $=((\forall x)^{-}\odot (\forall x)^{-})^{\sim}=\forall x\oplus \forall x$. \\ 
Conversely, if $(U_7)$ holds, then we have: \\
$\hspace*{2.00cm}$ $\exists(x\odot x)=\exists(x^{-}\oplus x^{-})^{\sim}=(\forall(x^{-}\oplus x^{-}))^{\sim}=
                   (\forall x^{-}\oplus \forall x^{-})^{\sim}$ \\
$\hspace*{3.50cm}$ $=((\exists x)^{-}\oplus (\exists x)^{-})^{\sim}=\exists x\odot \exists x$. \\
Thus, $(E_7)$ is equivalent to $(U_7)$. \\
We conclude that the axioms $(E_1)-(E_7)$ are equivalent to axioms $(U_1)-(U_7)$.                
\end{proof}

\begin{corollary} \label{qinv-30-20} The following hold: \\
$(1)$ There is a one-to-one correspondence between existential and universal quantifiers on $X;$ \\ 
$(2)$ There is a one-to-one correspondence between weak existential and weak universal quantifiers on $X$.  
\end{corollary}

Denote by: \\
$\hspace*{1cm}$ $\mathcal{QSM}(X)$ the set of all $(\exists,\forall)\in \mathcal{SM}(X)$ such that $\exists$ is an existential quantifier on $X$ or $\forall$ is a universal quantifier on $X$. \\
$\hspace*{1cm}$ $\mathcal{QSM}^{\mathcal{W}}(X)$ the set of all $(\exists,\forall)\in \mathcal{SM}(X)$ such that $\exists$ is a weak existential quantifier on $X$ or $\forall$ is a weak universal quantifier on $X$. \\
$\hspace*{1cm}$ $\mathcal{QSSM}(X)$ the set of all $(\exists,\forall)\in \mathcal{SSM}(X)$ such that $\exists$ is an existential quantifier on $X$ or $\forall$ is a universal quantifier on $X$. \\
$\hspace*{1cm}$ $\mathcal{QSSM}^{\mathcal{W}}(X)$ the set of all $(\exists,\forall)\in \mathcal{SSM}(X)$ such that $\exists$ is a weak existential quantifier on $X$ or $\forall$ is a weak universal quantifier on $X$. \\
Obviously, $\mathcal{QSM}(X)\subseteq \mathcal{QSM}^{\mathcal{W}}(X)$ and 
$\mathcal{QSSM}(X)\subseteq \mathcal{QSSM}^{\mathcal{W}}(X)$. 

\begin{proposition} \label{qinv-30-30} Consider the axioms: \\
$(E_4^{\prime})$ $\exists((\exists x)^{\sim}\ra y)=(\exists x)^{\sim}\ra \exists y$, 
                 $\exists((\exists x)^{-}\rs y)=(\exists x)^{-}\rs \exists y;$ \\
$(E_5^{\prime})$ $\exists(x^{\sim}\ra x)=(\exists x)^{\sim}\ra \exists x$,  
                 $\exists(x^{-}\rs x)=(\exists x)^{-}\rs \exists x;$ \\
$(U_6^{\prime})$ $\forall((\forall x)^{\sim}\ra y)=(\forall x)^{\sim}\ra \forall y$, 
                 $\forall((\forall x)^{-}\rs y)=(\forall x)^{-}\rs \forall y;$ \\
$(U_7^{\prime})$ $\forall(x^{\sim}\ra x)=(\forall x)^{\sim}\ra \forall x$,  
                 $\forall(x^{-}\rs x)=(\forall x)^{-}\rs \forall x$. \\
The axioms from the pairs $((E_4), (E_4^{\prime}))$, $((E_5), (E_5^{\prime}))$, $((U_6), (U_6^{\prime}))$, 
$((U_7), (U_7^{\prime}))$ are equivalent. 
\end{proposition}
\begin{proof} We use the definition of $\oplus$. \\
$(E_4)\Rightarrow (E_4^{\prime})$ 
$\exists((\exists x)^{\sim}\ra y)=\exists(y\oplus \exists x)=\exists y\oplus \exists x
                                 =(\exists x)^{\sim}\ra \exists x$. \\
$\exists((\exists x)^{-}\rs y)=\exists(\exists x\oplus y)=\exists x\oplus \exists y
                              =(\exists x)^{-}\rs \exists x$. \\
$(E_4^{\prime})\Rightarrow (E_4)$ 
$\exists(x\oplus \exists y)=\exists((\exists y)^{\sim}\ra x)=(\exists y)^{\sim}\ra \exists x
                           =\exists x\oplus \exists y$. \\
$\exists(\exists x\oplus y)=\exists((\exists y)^{-}\rs y)=(\exists y)^{-}\rs \exists y
                           =\exists x\oplus \exists y$. \\
$(E_5)\Rightarrow (E_5^{\prime})$
$\exists(x^{\sim}\ra x)=\exists(x\oplus x)=\exists x\oplus \exists x=(\exists x)^{\sim}\ra \exists x$. \\              $\exists(x^{-}\rs x)=\exists(x\oplus x)=\exists x\oplus \exists x=(\exists x)^{-}\rs \exists x$. \\                    $(E_5^{\prime})\Rightarrow (E_5)$           
$\exists(x\oplus x)=\exists(x^{\sim}\ra x)=(\exists x)^{\sim}\ra \exists x=\exists x\oplus \exists x$. \\
$(U_6)\Rightarrow  (U_6^{\prime})$
$\forall((\forall x)^{\sim}\ra y)=\forall(y\oplus \forall x)=\forall y\oplus \forall x=
         (\forall x)^{\sim}\ra \forall y$. \\
$\forall((\forall x)^{-}\rs y)=\forall(\forall x\oplus y)=\forall x\oplus \forall y=(\forall x)^{-}\rs \forall y$. \\
$(U_6^{\prime})\Rightarrow (U_6)$
$\forall(x\oplus \forall y)=\forall((\forall y)^{\sim}\ra x)=(\forall y)^{\sim}\ra \forall x)=
         \forall x\oplus \forall y$. \\ 
$\forall(\forall x\oplus y)=\forall((\forall x)^{-}\rs y)=(\forall x)^{\sim}\rs \forall y)=
         \forall x\oplus \forall y$. \\ 
$(U_7)\Rightarrow  (U_7^{\prime})$
$\forall(x^{\sim}\ra x)=\forall(x\oplus x)=\forall x\oplus \forall x=(\forall x)^{\sim}\ra \forall x$. \\
$\forall(x^{-}\rs x)=\forall(x\oplus x)=\forall x\oplus \forall x=(\forall x)^{-}\rs \forall x$. \\
$(U_7^{\prime})\Rightarrow (U_7)$
$\forall(x\oplus x)=\forall(x^{\sim}\ra x)=(\forall x)^{\sim}\ra \forall x=\forall x\oplus \forall x$. 
\end{proof}

\begin{corollary} \label{qinv-30-40} Let $\exists, \forall\in \mathcal{GM}(X)$. Then: \\
$(1)$ $\exists$ is an existential quantifier if it satisfies axioms $(E_1)$, $(E_2)$, $(E_3)$, $(E_4^{\prime})$, 
$(E_5^{\prime})$, $(E_6)$, $(E_7);$ \\
$(2)$ $\exists$ is a weak existential quantifier if it satisfies axioms $(E_1)$, $(E_2)$, $(E_3)$, $(E_4^{\prime})$, 
$(E_5^{\prime})$, $(E_6);$ \\
$(3)$ $\forall$ is a universal quantifier if it satisfies axioms $(U_1)$, $(U_2)$, $U_3)$, $(U_4)$, $(U_5)$,  $(U_6^{\prime})$, $(U_7^{\prime});$ \\
$(4)$ $\forall$ is a weak universal quantifier if it satisfies axioms $(U_1)$, $(U_2)$, $U_3)$, $(U_4)$, $(U_5)$,  $(U_6^{\prime})$. 
\end{corollary}

\begin{lemma} \label{qinv-40} Let $(\exists,\forall)\in \mathcal{QSM}^{\mathcal{W}}(X)$. 
Then the following hold, for all $x\in X:$ \\
$(1)$ $\exists x^{-}=(\forall x)^{-}$, $\exists x^{\sim}=(\forall x)^{\sim};$ \\ 
$(2)$ $\forall x^{-}=(\exists x)^{-}$, $\forall x^{\sim}=(\exists x)^{\sim};$ \\ 
$(3)$ $\exists\exists x=\exists x$, $\forall\forall x=\forall x$.   
\end{lemma}
\begin{proof}
$(1)$, $(2)$ follow by Proposition \ref{gm-50}. \\
$(3)$ Using $(E_4^{\prime})$ and $(E_1)$ we have 
$\exists\exists x=\exists((\exists x)^{-\sim})=\exists((\exists x)^{-}\rs d)=(\exists x)^{-}\rs \exists d=
(\exists x)^{-}\rs d=(\exists x)^{-\sim}=\exists x$. 
Similarly, by $(U_6^{\prime})$ and $(U_2)$, $\forall\forall x=\forall((\forall)^{-\sim})=
\forall((\forall x)^{-}\rs d)=(\forall x)^{-}\rs \forall d=(\forall x)^{-}\rs d=(\forall x)^{-\sim}=\forall x$. 
\end{proof}

\begin{lemma} \label{qinv-50} Let $(\exists,\forall)\in \mathcal{QSSM}^{\mathcal{W}}(X)$. 
Then the following hold, for all $x\in X:$ \\
$(1)$ $\exists\forall x=\forall x$, $\forall\exists x=\exists x;$ \\
$(2)$ $\exists(\exists x)^{-}=(\exists x)^{-}$, $\exists(\exists x)^{\sim}=(\exists x)^{\sim};$ \\
$(3)$ $\forall(\forall x)^{-}=(\forall x)^{-}$, $\forall(\forall x)^{\sim}=(\forall x)^{\sim}$. 
\end{lemma}
\begin{proof}
$(1)$ follows by Proposition \ref{gm-60}$(1)$, since $\exists$ and $\forall$ are strong synchronized. \\
$(2)$ Using $(1)$ and Lemma \ref{qinv-40}, we get 
$\exists(\exists x)^{-}=\exists\forall x^{-}=\forall x^{-}=(\exists x)^{-}$ and 
$\exists(\exists x)^{\sim}=\exists\forall x^{\sim}=\forall x^{\sim}=(\exists x)^{\sim}$. \\
$(3)$ Similarly, $\forall(\forall x)^{-}=\forall\exists x^{-}=\exists x^{-}=(\forall x)^{-}$ and 
$\forall(\forall x)^{\sim}=\forall\exists x^{\sim}=\exists x^{\sim}=(\forall x)^{\sim}$. 
\end{proof}

\begin{proposition} \label{qinv-60} Let $(\exists,\forall)\in \mathcal{QSSM}^{\mathcal{W}}(X)$. 
Then the following hold, for all $x, y\in X:$ \\
$(1)$ $\exists(\forall x\ra y)=\forall x\ra \exists y$, $\exists(\forall x\rs y)=\forall x\rs \exists y;$ \\ 
$(2)$ $\exists(x\ra \forall y)=\forall x\ra \forall y$, $\exists(x\rs \forall y)=\forall x\rs \forall y;$ \\
$(3)$ $\exists(\exists x\ra y)=\exists x\ra \exists y$, $\exists(\exists x\rs y)=\exists x\rs \exists y;$ \\ 
$(4)$ $\forall(\exists x\ra y)=\exists x\ra \forall y$, $\forall(\exists x\rs y)=\exists x\rs \forall y;$ \\ 
$(5)$ $\forall(x\ra \exists y)=\exists x\ra \exists y$, $\forall(x\rs \exists y)=\exists x\rs \exists y;$ \\
$(6)$ $\forall(\forall x\ra y)=\forall x\ra \forall y$, $\forall(\forall x\rs y)=\forall x\rs \forall y;$ \\
$(7)$ $\forall(x\ra \forall y)=\exists x\ra \forall y$, $\forall(x\rs \forall y)=\exists x\rs \forall y$. 
\end{proposition}
\begin{proof} 
We apply Propositions \ref{inv-qb-50}, \ref{qinv-30-30} and Lemma \ref{qinv-50}. \\
$(1)$ Using $(E_4^{\prime})$ we have: \\
$\hspace*{2cm}$ $\exists(\forall x\ra y)=\exists((\exists x^{-})^{\sim}\ra y)
                =(\exists x^{-})^{\sim}\ra \exists y=\forall x\ra \exists y$. \\
$\hspace*{2cm}$ $\exists(\forall x\rs y)=\exists((\exists x^{\sim})^{-}\rs y)
                =(\exists x^{\sim})^{-}\rs \exists y=\forall x\rs \exists y$. \\
$(2)$ Applying $(E_4^{\prime})$ we get: \\
$\hspace*{2.00cm}$ $\exists(x\ra \forall y)=\exists((\forall y)^{-}\rs x^{-})=\exists((\exists\forall y)^{-}\rs x^{-})
                    =(\exists\forall y)^{-}\rs \exists x^{-}$ \\
$\hspace*{3.85cm}$ $=(\forall y)^{-}\rs (\forall x)^{-}=\forall x\ra \forall y$. \\
$\hspace*{2.00cm}$ $\exists(x\rs \forall y)=\exists((\forall y)^{\sim}\ra x^{\sim})
                    =\exists((\exists\forall y)^{\sim}\ra x^{\sim})=(\exists\forall y)^{\sim}\ra \exists x^{\sim}$ \\
$\hspace*{3.85cm}$ $=(\forall y)^{\sim}\ra (\forall x)^{\sim}=\forall x\rs \forall y$. \\
$(3)$ By $(E_4^{\prime})$ and $(2)$ we have: \\ 
$\hspace*{2.00cm}$ $\exists(\exists x\ra y)=\exists(y^{-}\rs (\exists x)^{-})=\exists(y^{-}\rs \forall x^{-})
                     =\forall y^{-}\rs \forall x^{-}$ \\
$\hspace*{3.85cm}$ $=(\exists y)^{-}\rs (\exists x)^{-}=\exists x\ra \exists y$. \\ 
$\hspace*{2.00cm}$ $\exists(\exists x\rs y)=\exists(y^{\sim}\ra (\exists x)^{\sim})
                     =\exists(y^{\sim}\ra \forall x^{\sim})=\forall y^{\sim}\ra \forall x^{\sim}$ \\
$\hspace*{3.85cm}$ $=(\exists y)^{\sim}\ra (\exists x)^{\sim}=\exists x\rs \exists y$. \\
$(4)$ Using $(U_6^{\prime})$ we get: \\ 
$\hspace*{2.00cm}$ $\forall(\exists x\ra y)=\forall((\forall x^{-})^{\sim}\ra y)=(\forall x^{-})^{\sim}\ra \forall y 
                                           =\exists x\ra \forall y$. \\
$\hspace*{2.00cm}$ $\forall(\exists x\rs y)=\forall((\forall x^{\sim})^{-}\rs y)=(\forall x^{\sim})^{-}\rs \forall y
                                           =\exists x\rs \forall y$. \\
$(5)$ Applying $(U_6^{\prime})$ we have: \\ 
$\hspace*{2.00cm}$ $\forall(x\ra \exists y)=\forall((\exists y)^{-}\rs x^{-})
                     =\forall((\forall\exists y)^{-}\rs x^{-})=(\forall\exists y)^{-}\rs \forall x^{-}$ \\
$\hspace*{3.85cm}$ $=(\exists y)^{-}\rs (\exists x)^{-}=\exists x\ra \exists y$. \\
$\hspace*{2.00cm}$ $\forall(x\rs \exists y)=\forall((\exists y)^{\sim}\ra x^{\sim})
                     =\forall((\forall\exists y)^{\sim}\ra x^{\sim})=(\forall\exists y)^{\sim}\ra \forall x^{\sim}$ \\
$\hspace*{3.85cm}$ $=(\exists y)^{\sim}\ra (\exists x)^{\sim}=\exists x\rs \exists y$. \\
$(6)$ Taking into consideration $(5)$ we get: \\ 
$\hspace*{2.00cm}$ $\forall(\forall x\ra y)=\forall(y^{-}\rs (\forall x)^{-})
                    =\forall(y^{-}\rs \exists x^{-})=\exists y^{-}\rs \exists x^{-}$ \\        
$\hspace*{3.85cm}$ $=(\forall y)^{-}\rs (\forall x)^{-}=\forall x\ra \forall y$. \\             
$\hspace*{2.00cm}$ $\forall(\forall x\rs y)=\forall(y^{\sim}\ra (\forall x)^{\sim})
                    =\forall(y^{\sim}\ra \exists x^{\sim})=\exists y^{\sim}\ra \exists x^{\sim}$ \\ 
$\hspace*{3.85cm}$ $=(\forall y)^{\sim}\ra (\forall x)^{\sim}=\forall x\rs \forall y$. \\
$(7)$ Using Lemma \ref{qinv-50} and $(5)$ we have: \\
$\hspace*{2.00cm}$ $\forall(x\ra \forall y)=\forall(x\ra \exists\forall y)=\exists x\ra \exists\forall y
                                           =\exists x\ra \forall y$. \\
$\hspace*{2.00cm}$ $\forall(x\rs \forall y)=\forall(x\rs \exists\forall y)=\exists x\rs \exists\forall y
                                           =\exists x\rs \forall y$. 
\end{proof}

\begin{proposition} \label{qinv-60-10} Let $(\exists,\forall)\in \mathcal{QSSM}^{\mathcal{W}}(X)$. 
Then the following hold, for all $x, y\in X:$ \\
$(1)$ $\forall(\forall x\ra \forall y)=\forall x\ra \forall y$, 
      $\forall(\forall x\rs \forall y)=\forall x\rs \forall y;$ \\ 
$(2)$ $\exists(\exists x\ra \exists y)=\exists x\ra \exists y$, 
      $\exists(\exists x\rs \exists y)=\exists x\rs \exists y;$ \\
$(3)$ $\forall(\forall x\odot \forall y)=\forall x\odot \forall y;$ \\
$(4)$ $\exists(\exists x\odot \exists y)=\exists x\odot \exists y;$ \\ 
$(5)$ $\forall((x\ra \forall y)\rs \forall y)=(\forall x\ra \forall y)\rs \forall y$, 
      $\forall((x\rs \forall y)\ra \forall y)=(\forall x\rs \forall y)\ra \forall y$.   
\end{proposition}
\begin{proof} We will apply Lemma \ref{qinv-40} and Proposition \ref{qinv-60}. \\
$(1)$ By Proposition \ref{qinv-60}$(6)$ replacing $y$ by $\forall y$. \\
$(2)$ By Proposition \ref{qinv-60}$(3)$ replacing $y$ by $\exists y$. \\
$(3)$ By $(U_4)$ replacing $x$ by $\forall x$. \\
$(4)$ By $(E_6)$ replacing $x$ by $\exists x$. \\
$(5)$ Applying Proposition \ref{qinv-60}$(7)$,$(2)$ we get: \\
$\hspace*{2cm}$ $\forall((x\ra \forall y)\ra \forall y)=\exists(x\ra \forall y)\ra \forall y
                    =(\forall x\ra \forall y)\ra \forall y$. \\
Similarly, $\forall((x\rs \forall y)\ra \forall y)=(\forall x\rs \forall y)\ra \forall y$.                     
\end{proof}

The monadic operators on unital quantum B-algebras $(X,\le,\ra,\rs,u)$ were recently defined in \cite{Ciu36} . 
In what follows, we consider the monadic operators on weakly involutive unital quantum B-algebras 
$(X,\le,\ra,\rs,d,u)$. 

\begin{definition} \label{qinv-70} $\rm($\cite{Ciu36}$\rm)$ 
The structure $(X,\exists,\forall)$ is called a \emph{monadic quantum B-algebra} if the following conditions are satisfied, for all $x, y\in X:$ \\
$(MQB_1)$ $x\le \exists x;$ \\
$(MQB_2)$ $\forall x\le x;$ \\
$(MQB_3)$ $\forall(x\ra \exists y)=\exists x\ra \exists y$, $\forall(x\rs \exists y)=\exists x\rs \exists y;$ \\
$(MQB_4)$ $\forall(\exists x\ra y)=\exists x\ra \forall y$, $\forall(\exists x\rs y)=\exists x\rs \forall y;$ \\
$(MQB_5)$ $\exists\forall x=\forall x;$ \\
$(MQB_6)$ $\exists u=\forall u=u$. 
\end{definition}

If $(X,\exists,\forall)$ is monadic quantum B-algebra, then the pair $(\exists, \forall)$ is said to be a 
\emph{monadic operator} on $X$. 

\begin{definition} \label{qinv-70-10}
A monadic operator $(\exists, \forall)$ on $X$ is called \emph{strict} if $\exists d=\forall d=d$. 
\end{definition}

Denote by $\mathcal{MOP}(X)$ and $\mathcal{MOP}^{\mathcal{S}}(X)$ the set of all monadic operators and the 
set of all strict monadic operators on $X$, respectively. 

\begin{theorem} \label{qinv-80} $\mathcal{QSSM}^{\mathcal{W}}(X)\subseteq \mathcal{MOP}^{\mathcal{S}}(X)$.
%Let $(\exists,\forall)\in \mathcal{QSSM}(X)$. Then $(X,\exists,\forall)$ is monadic quantum B-algebra. 
\end{theorem}
\begin{proof}
Axioms $(MQB_1)-(MQB_5)$ follow from $(E_3)$, $(U_3)$, Proposition \ref{qinv-60}$(5)$,$(4)$ and 
Lemma \ref{qinv-50}$(1)$, respectively. Axiom $(MQB_6)$ follows from $(E_2)$ and $(U_1)$. 
Since by $(E_1)$ and $(U_2)$, $(\exists,\forall)$ is strict, we get  
$\mathcal{QSSM}^{\mathcal{W}}(X)\subseteq \mathcal{MOP}^{\mathcal{S}}(X)$.  
\end{proof}

\begin{corollary} \label{qinv-80-10} 
$\mathcal{QSSM}(X)\subseteq \mathcal{QSSM}^{\mathcal{W}}(X)\subseteq \mathcal{MOP}^{\mathcal{S}}(X)
                  \subseteq \mathcal{MOP}(X)$.
\end{corollary}

\begin{example} \label{qinv-90} 
Let $X=\{0,u,a,b,1\}$ be a poset with $\le$ defined by $0\le u\le a\le 1$ and $0\le b\le 1$. 
Define the operation $\ra$ on $X$ by the following table: 
\[
\begin{array}{c|cccccc}
\ra & 0 & u & a & b & 1 \\ \hline
0   & 1 & 1 & 1 & 1 & 1 \\
u   & 0 & u & a & b & 1 \\
a   & 0 & 0 & u & b & 1 \\
b   & b & b & b & 1 & 1 \\
1   & 0 & 0 & 0 & b & 1 
\end{array}
.
\]
Then, $(X,\le,\ra,d,u)$ is a commutative involutive unital quantum B-algebra, with $d=a$ $\rm($\cite{Will1}$\rm)$. 
Moreover, $X$ is a quantum B(pP)-algebra with the operation $\odot$ defined below.
\[
\begin{array}{c|cccccc}
\odot & 0 & u & a & b & 1 \\ \hline
0     & 0 & 0 & 0 & 0 & 0 \\
u     & 0 & u & a & b & 1 \\
a     & 0 & a & 1 & b & 1 \\
b     & 0 & b & b & 0 & b \\
1     & 0 & 1 & 1 & b & 1 
\end{array}
.
\]
Consider the maps $\exists_i, \forall_i:X\longrightarrow X$, $i=1,...,6$, given in the table below:
\[
\begin{array}{c|cccccc}
 x          & 0 & u & a & b & 1 \\ \hline
\exists_1 x & 0 & u & a & b & 1 \\
\forall_1 x & 0 & u & a & b & 1 \\ \hline
\exists_2 x & 0 & u & a & 1 & 1 \\
\forall_2 x & 0 & u & a & 0 & 1 \\ \hline
\exists_3 x & 0 & u & 1 & 1 & 1 \\
\forall_3 x & 0 & u & 0 & 0 & 1 \\ \hline
\exists_4 x & 0 & u & 1 & b & 1 \\
\forall_4 x & 0 & u & 0 & b & 1 \\ \hline
\exists_5 x & 0 & u & 1 & 1 & 1 \\
\forall_5 x & 0 & u & u & 0 & 1 \\ \hline
\exists_6 x & 0 & u & 1 & b & 1 \\
\forall_6 x & 0 & u & u & b & 1 \\ 
\end{array}
.   
\]
One can check that $\mathcal{QSSM}(X)=\{(\exists_1,\forall_1)\}$, 
$\mathcal{QSSM}^{\mathcal{W}}(X)=\mathcal{MOP}^{\mathcal{S}}(X)=\{(\exists_1,\forall_1), (\exists_2,\forall_2)\}$,  
$\mathcal{MOP}(X)=\{(\exists_i,\forall_i)\mid i=1,...,6\}$.
\end{example}

$\vspace*{5mm}$

\section{Quantifiers on weakly involutive integral quantum B-algebras}

In this section, we investigate the relationship between quantifiers on weakly involutive integral quantum B-algebras and  quantifiers on other related algebraic structures, such as pseudo MV-algebras and bounded Wajsberg hoops. 
We prove that the quantifiers on bounded sup-commutative pseudo BCK algebras are also quantifiers on the 
corresponding pseudo MV-algebras. We also give a condition for the universal quantifiers of a sup-commutative BCK-algebra to be universal quantifiers on a bounded Wajsberg hoop. \\ 
According to Proposition \ref{psBE-70}, any integral quantum B-algebra is a pseudo BCK-algebra, so that, 
by Proposition \ref{inv-qb-60}, a weakly involutive pseudo BCK-algebra $(X,\le,\ra,\rs,d,u)$ is in fact an 
involutive pseudo BCK-algebra $(X,\le,\ra,\rs,0,1)$ defined in \cite[Def. 4]{Ior16}. 
It follows that the results proved in the previous sections for weakly involutive quantum B-algebras are also 
valid in the case of involutive pseudo BCK-algebras investigated in \cite{Ior16}. 

\begin{lemma} \label{qinv-iqb-10} Let $X$ be an involutive pseudo BCK-algebra. The quantifiers $\exists$, 
$\forall$ on $X$ are isotone. 
\end{lemma}
\begin{proof}
If $x, y\in X$ such that $x\le y$, then $\forall x\le x\le y$ and $x\le y\le \exists y$. 
Applying Proposition  \ref{qinv-60}$(6)$,$(5)$ and $(BCI_5)$ we get: 
$1=\forall 1=\forall(\forall x\ra y)=\forall x\ra \forall y$ and 
$1=\forall 1=\forall(x\ra \exists y)=\exists x\ra \exists y$, respectively. 
It follows that $\forall x\le \forall y$ and $\exists x\le \exists y$, hence the quantifiers are isotone. 
\end{proof}

\begin{proposition} \label{qinv-60-20} Let $X$ be an involutive pseudo BCK-algebra and let 
$(\exists,\forall)\in \mathcal{QSSM}^{\mathcal{W}}(X)$. Then the following hold, for all $x, y\in X:$ \\
$(1)$ $\forall(x\ra y)\rs (\forall x\ra \forall y)=1$, $\forall(x\rs y)\ra (\forall x\rs \forall y)=1;$ \\ 
$(2)$ $\forall((x\ra \forall y)\rs \forall x)\le \forall((x\ra \forall y)\rs x)$, 
      $\forall((x\rs \forall y)\ra \forall x)\le \forall((x\rs \forall y)\ra x)$.  
\end{proposition}
\begin{proof} 
$(1)$ From $\forall x\le x$ we have $\forall(x\ra y)\le x\ra y\le \forall x\ra y$, that is, 
$\forall(x\ra y)\rs (\forall x\ra y)=1$. Applying Lemma \ref{qinv-50} and Proposition \ref{qinv-60}$(4)$ we get \\
$\hspace*{2cm}$ $1=\forall 1=\forall(\forall(x\ra y)\rs (\forall x\ra y))=
                   \forall(\exists\forall(x\ra y)\rs (\forall x\ra y))$ \\
$\hspace*{2.3cm}$ $=\exists\forall(x\ra y)\rs \forall(\forall x\ra y)=
                \forall(x\ra y)\rs \forall(\exists\forall x\ra y)$ \\
$\hspace*{2.3cm}$ $=\forall(x\ra y)\rs (\exists\forall x\ra \forall y)=\forall(x\ra y)\rs (\forall x\ra \forall y)$. \\
Similarly, $\forall(x\rs y)\ra (\forall x\rs \forall y)=1$. \\
$(2)$ From $\forall x\le x$ we get $(x\ra \forall y)\rs \forall x\le (x\ra \forall y)\rs x$, so that 
$\forall((x\ra \forall y)\rs \forall x)\le \forall((x\ra \forall y)\rs x)$. 
Applying Proposition \ref{qinv-60}$(7)$,$(2)$ we have: \\ 
$\hspace*{2cm}$ $\forall((x\ra \forall y)\rs \forall x)=\exists(x\ra \forall y)\rs \forall x
                    =(\forall x\ra \forall y)\rs \forall x$. \\
It follows that $(\forall x\ra \forall y)\rs \forall x\le \forall((x\ra \forall y)\rs x)$. 
Similarly, $\forall((x\rs \forall y)\ra \forall x)\le \forall((x\rs \forall y)\ra x)$. 
\end{proof}

A pseudo BCK-algebra $(X,\le,\ra,\rs,1)$ is said to be \emph{sup-commutative} (\cite{Geo15}) if it satisfies the following conditions, for all $x, y\in X:$ \\
$(comm_1)$ $(x\ra y)\rs y=(y\ra x)\rs x$, \\
$(comm_2)$ $(x\rs y)\ra y=(y\rs x)\ra x$. \\
If $X$ is sup-commutative, then $(x\ra y)\rs y=(x\rs y)\ra y$ (\cite[Cor. 1.2]{Ciu2}). 

\begin{remark} \label{qinv-iqb-10-10} 
The bounded sup-commutative pseudo BCK-algebras form a proper subclass of involutive pseudo BCK-algebras. \\
Indeed, if $(X,\le,\ra,\rs,0,1)$ is a bounded sup-commutative pseudo BCK-algebra, taking $y:=0$ in $(comm_1)$ and $(comm_2)$ we get $x^{-\sim}=x^{\sim-}=x$, so that $X$ is an involutive pseudo BCK-algebra. 
The converse is not always true, as we can see in the following example. 
Consider the pseudo-BCK algebra $(X,\le,\ra,\rs,0,1)$ where the operations $\ra$ and $\rs$ on $X=\{0,a,b,c,d,1\}$ 
are defined as follows:
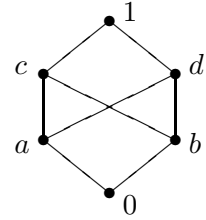
\[
\begin{array}{c|cccccc}
\ra & 0 & a & b & c & d & 1 \\ \hline
0 & 1 & 1 & 1 & 1 & 1 & 1 \\
a & d & 1 & d & 1 & 1 & 1 \\
b & c & c & 1 & 1 & 1 & 1 \\
c & a & a & d & 1 & d & 1 \\
d & b & c & b & c & 1 & 1 \\
1 & 0 & a & b & c & d & 1
\end{array}
\hspace{10mm}
\begin{array}{c|cccccc}
\rs & 0 & a & b & c & d & 1 \\ \hline
0 & 1 & 1 & 1 & 1 & 1 & 1 \\
a & c & 1 & c & 1 & 1 & 1 \\
b & d & d & 1 & 1 & 1 & 1 \\
c & b & d & b & 1 & d & 1 \\
d & a & a & c & c & 1 & 1 \\
1 & 0 & a & b & c & d & 1
\end{array}
\qquad\quad
\begin{picture}(50,-70)(0,30)
\drawline(0,25)(25,5)(50,25)(0,50)(0,25)(50,50)(50,25)
\drawline(0,50)(25,70)(50,50)
\put(0,25){\circle*{4}}
\put(25,5){\circle*{4}}
\put(50,25){\circle*{4}}
\put(0,50){\circle*{4}}
\put(50,50){\circle*{4}}
\put(25,70){\circle*{4}}
\put(30,-3){$0$}
\put(30,70){$1$}
\put(-11,20){$a$}
\put(-11,50){$c$}
\put(55,20){$b$}
\put(55,50){$d$}
\end{picture}
\]
Then $X$ is involutive, but not sup-commutative (\cite{Ciu30}) (for example, $(a\ra b)\rs b=c\neq d=((b\ra a)\rs a$). 
\end{remark}

If $X$ is a bounded sup-commutative pseudo BCK-algebra, then $(X,\le)$ is a lattice with $\vee$ and $\wedge$ 
defined by $x\vee y=(x\ra y)\rs y=(x\rs y)\ra y$, $x\wedge y=(x^{-}\vee y^{-})^{\sim}=(x^{\sim}\vee y^{\sim})^{-}$, 
for all $x, y\in A$. 

\begin{lemma} \label{qinv-iqb-20} Let $X$ be abounded sub-commutative pseudo BCK-algebra. Then 
$\forall(x\wedge y)=\forall x \wedge \forall y$ and $\exists(x\vee y)=\exists x\vee \exists y$, 
for all $x, y\in X$. 
\end{lemma}
\begin{proof} Similarly to \cite[Prop. 5.6,5.7]{Ciu35}. 
\end{proof}

The non-commutative generalizations of MV-algebras called pseudo MV-algebras were introduced by G. Georgescu 
and A. Iorgulescu (\cite{Geo2}) and independently by J. Rach{\accent23u}nek (\cite{Rac2}) under the name of 
generalized MV-algebras. 
A pseudo MV-algebra is a structure $(X, \oplus, \odot, ^{-}, ^{\sim}, 0, 1)$ of type $(2,2,1,1,0,0)$ such 
that the following axioms hold for all $x, y, z \in X:$ 
$(psMV_1)$ $x\oplus (y \oplus z) = (x \oplus y) \oplus z;$ 
$(psMV_2)$ $x \oplus 0 = 0 \oplus x = x;$ 
$(psMV_3)$ $x \oplus 1 = 1 \oplus x =1;$ 
$(psMV_4)$ $1^{-}=0$, $1^{\sim}=0;$ 
$(psMV_5)$ $(x^{-} \oplus y^{-})^{\sim}=(x^{\sim} \oplus y^{\sim})^{-};$ 
$(psMV_6)$ $x \oplus (x^{\sim} \odot y)=y \oplus (y^{\sim} \odot x) = (x \odot y^{-}) \oplus y =
            (y \odot x^{-}) \oplus x;$ 
$(psMV_7)$ $x \odot (x^{-} \oplus y) = (x \oplus y^{\sim}) \odot y;$ 
$(psMV_8)$ $(x^{-})^{\sim}=x,$ 
where $x \odot y:= (y^{-} \oplus x^{-})^{\sim}$. \\
We define $x \leq y$ iff $x^{-}\oplus y=1$ and ``$\leq$" defines an order relation on $X$. 
Moreover, $(X,\le)$ is a distributive lattice with the lattice operations defined as below: \\
$\hspace*{2cm}$ $x \vee y:= x \oplus x^{\sim} \odot y=y \oplus y^{\sim} \odot x = 
                 x \odot y^{-} \oplus y = y \odot x^{-} \oplus x,$ \\
$\hspace*{2cm}$ $x \wedge y:= x \odot (x^{-} \oplus y) = y \odot (y^{-} \oplus x)=
                (x \oplus y^{\sim}) \odot y =(y \oplus x^{\sim}) \odot x.$ \\
It was proved in \cite[Th. 3.7]{Ior1} that the bounded sup-commutative pseudo BCK-algebras coincide (are categorically isomorphic) with pseudo MV-algebras. 
Namely, given a bounded sup-commutative pseudo BCK-algebra $(X,\le,\ra,\rs,0,1)$, then the structure 
$(X,\odot,\oplus,^{-},^{\sim},0,1)$ defined by $x^{-}=x\ra 0$, $x^{\sim}=x\rs 0$, 
$x\odot y=(x\ra y^{-})^{\sim}=(y\rs x^{\sim})^{-}$, $x\oplus y=(y^{-}\odot x^{-})^{\sim}=(y^{\sim}\odot x^{\sim})^{-}$, is a pseudo MV-algebra. Conversely, given a pseudo MV-algebra $(X,\odot,\oplus,^{-},^{\sim},0,1)$, then the structure $(X,\le,\ra,\rs,0,1)$ defined by 
$x\ra y=y\oplus x^{-}=(x\odot y^{\sim})^{-}$, $x\rs y=x^{\sim}\oplus y=(y^{-}\odot x)^{\sim}$, $x \leq y$ iff $x^{-}\oplus y=1$ is a bounded sup-commutative pseudo BCK-algebra. \\
The universal and existential quantifiers on a pseudo MV-algebra $X$ were defined and investigated in \cite{Rac1} 
as follows. 
A universal quantifier on $X$ is a mapping $\forall:X\longrightarrow X$ satisfying the following conditions, 
for all $x, y\in X:$ 
$(MVU_1)$ $\forall x\le x;$ 
$(MVU_2)$ $\forall(x\wedge y)=\forall x\wedge \forall y;$ 
$(MVU_3)$ $\forall((\forall x)^{-})=(\forall x)^{-}$, $\forall((\forall x)^{\sim})=(\forall x)^{\sim};$ 
$(MVU_4)$ $\forall(\forall x\odot \forall y)=\forall x\odot \forall y;$ 
$(MVU_5)$ $\forall(x \odot x)=\forall x\odot \forall x;$ 
$(MVU_6)$ $\forall(x \oplus x)=\forall x\oplus \forall x$. 
An existential quantifier on $X$ is a mapping $\exists:X\longrightarrow X$ satisfying the following conditions, 
for all $x, y\in X:$ 
$(MVE_1)$ $x\le \exists x;$ 
$(MVE_2)$ $\exists(x\vee y)=\exists x\vee \exists y;$ 
$(MVE_3)$ $\exists((\exists x)^{-})=(\exists x)^{-}$, $\exists((\exists x)^{\sim})=(\exists x)^{\sim};$ 
$(MVE_4)$ $\exists(\exists x\odot \exists y)=\exists x\odot \exists y;$ 
$(MVE_5)$ $\exists(x \odot x)=\exists x\odot \exists x;$ 
$(MVE_6)$ $\exists(x \oplus x)=\exists x\oplus \exists x$. 

\begin{proposition} \label{qinv-iqb-30} Let $X$ be a bounded sup-commutative pseudo BCK-algebra. 
The existential quantifier from Definition \ref{qinv-10} is an existential quantifier on the coresponding 
pseudo MV-algebra. 
\end{proposition}
\begin{proof}
Axiom $(MVE_1)$ is $(E_3)$, while axiom $(MVE_2)$ follows by Lemma \ref{qinv-iqb-20}, and axiom $(MVE_3)$ 
follows by Lemma \ref{qinv-50}$(2)$. 
Since axioms $(MVE_4)$, $(MVE_5)$ and $(MVE_6)$ are in fact $(E_6)$, $(E_7)$ and $(E_5)$, we conclude that 
$\exists$ is a universal quantifier on the pseudo MV-algebra $X$. 
\end{proof}

\begin{proposition} \label{qinv-iqb-40} Let $X$ be a bounded sup-commutative pseudo BCK-algebra. 
The universal quantifier $\forall$ from Definition \ref{qinv-20} is a universal quantifier on the coresponding 
pseudo MV-algebra. 
\end{proposition}
\begin{proof}
Axiom $(MVU_1)$ is $(U_3)$, while axiom $(MVU_2)$ follows by Lemma \ref{qinv-iqb-20}, and axiom $(MVU_3)$ 
follows by Lemma \ref{qinv-50}$(3)$. 
Since axioms $(MVU_4)$, $(MVU_5)$ and $(MVU_6)$ are in fact $(U_4)$, $(U_5)$ and $(U_7)$, we conclude that 
$\forall$ is a universal quantifier on the pseudo MV-algebra $X$. 
\end{proof}

It follows that the results of this paper can be also applied to pseudo MV-algebras. 

A \emph{hoop} is an algebra $(H,\odot,\ra,1)$ of the type $(2,2,0)$ such that $(H,\odot,1)$ is a 
commutative monoid satisfying the following axioms, for all $x, y, z \in H:$ 
$(H_1)$ $x\ra x=1;$ 
$(H_2)$ $x\odot (x\ra y)=y\odot (y\ra x);$ 
$(H_3)$ $(x\odot y) \ra z=x\ra (y\ra z)$. 
A hoop $H$ satisfying the axiom $(WH)$ $(x\ra y)\ra y=(y\ra x)\ra x$, for all $x, y\in H$ is called a 
\emph{Wajsberg hoop}. Obviously, a bounded Wajsberg hoop is involutive. 
By \cite[Prop. 5.4]{Ciu14}, every hoop $(H,\odot,\ra,1$ is a BCK(P)-algebra, and so, every bounded Wajsberg hoop 
is an involutive BCK-algebra. 
As a consequence, the results proved for involutive pseudo BCK-algebras are also valid for the case of 
bounded Wajsberg hoops. 
Given a hoop $(H,\odot,\ra,1)$, the operation $\oplus$ was defined in \cite{Cim1} as follows: \\
$(WS)$ $x\oplus y=(x\ra x\odot y)\ra y$, for all $x, y\in H$. 

\begin{lemma} \label{qinv-iqb-50} If $(H,\odot,\ra,0,1)$ is a bounded Wajsberg hoop, then the definitions $(S)$ and 
$(WS)$ of $\oplus$ are equivalent, that is  
$x\oplus y=(x\ra x\odot y)\ra y=x^{-}\ra y=y^{-}\ra x$, for all $x, y\in H$.
\end{lemma}
\begin{proof} Applying Propositions \ref{inv-qb-50} and \ref{inv-qb-100}$(4)$, we get: \\
$\hspace*{2.00cm}$ $(x\ra x\odot y)\ra y=y^{-}\ra (x\ra x\odot y)^{-}=y^{-}\ra ((x\odot y)^{-}\ra x^{-})$ \\ 
$\hspace*{4.90cm}$ $=y^{-}\ra ((x\ra y^{-})\ra x^{-})^{-}=y^{-}\ra ((y\ra x^{-})\ra x^{-})^{-}$ \\
$\hspace*{4.90cm}$ $=y^{-}\ra (y\vee x^{-})^{-}=y^{-}\ra y^{-}\wedge x$ \\
$\hspace*{4.90cm}$ $=(y^{-}\ra y^{-})\wedge (y^{-}\ra x)=y^{-}\ra x=x^{-}\ra y$.  
\end{proof}

The notion of a monadic Wajsberg hoop was introduced in \cite{Cim1} as an algebra $(H,\odot,\ra,\forall,1)$, where 
$(H,\odot,\ra,1)$ is a Wajsberg hoop and the universal quantifier $\forall$ satisfies the following axioms: 
$(MH_1)$ $\forall 1=1;$ 
$(MH_2)$ $\forall x\ra x=1;$ 
$(MH_3)$ $\forall((x\ra \forall y)\ra \forall y)=(\forall x\ra \forall y)\ra \forall y;$ 
$(MH_4)$ $\forall(x\ra y)\ra (\forall x\ra \forall y)=1;$ 
$(MH_5)$ $\forall(\forall x\ra \forall y)=\forall x\ra \forall y;$ 
$(MH_6)$ $\forall(x\odot x)=\forall x\odot \forall x;$ 
$(MH_7)$ $\forall((x\ra \forall y)\ra x)=(\forall x\ra \forall y)\ra \forall x;$ 
$(MH_8)$ $\forall(x\wedge y)=\forall x\wedge \forall y;$ 
$(MH_9)$ $\forall(\forall x\odot \forall y)=\forall x\odot \forall y$. 
As we mentioned, a bounded Wajsberg hoop $(H,\odot,\ra,\ra,0,1)$ is an involutive pseudo BCK-algebra.  

\begin{proposition} \label{qinv-iqb-50-10} 
Let $(X,\le,\ra,0,1)$ be a bounded sup-commutative BCK-algebra and let $\forall:X\longrightarrow X$ be the 
universal quantifier from Definition \ref{qinv-20}. 
Then $\forall$ satisfies the axioms $(MH_1)-(MH_6)$ and $(MH_8)-(MH_9)$, but not $(MH_7)$. 
\end{proposition} 
\begin{proof}
Axioms $(MH_1)$, $(MH_2)$ and $(MH_6)$ are in fact the axioms $(U_1)$, $(U_3)$ and $(U_5)$, while axioms $(MH_3)$, 
$(MH_4)$, $(MH_5)$,$(MH_9)$ and $(MH_8)$ are consequences of Propositions \ref{qinv-60-10}$(5)$, 
\ref{qinv-60-20}$(1)$, \ref{qinv-60-10}$(1)$,$(3)$ and Lemma  \ref{qinv-iqb-20}, respectively. 
Let $X=\{0,a,1\}$ be a poset with $\le$ defined by $0\le a\le 1$. 
Define the operation $\ra$ on $X$ by the following table: 
\[
\begin{array}{c|cccccc}
\ra & 0 & a & 1 \\ \hline
0   & 1 & 1 & 1 \\
a   & a & 1 & 1 \\
1   & 0 & a & 1 
\end{array}
.
\]
Then, $(X,\le,\ra,0,1)$ is a sup-commutative BCK-algebra $\rm($\cite{Will1}$\rm)$. 
Consider the maps $\exists_i, \forall_i:X\longrightarrow X$, $i=1,2$, given in the table below:
\[
\begin{array}{c|cccccc}
 x          & 0 & a & 1 \\ \hline
\exists_1 x & 0 & 1 & 1 \\
\forall_1 x & 0 & 0 & 1 \\ \hline
\exists_2 x & 0 & a & 1 \\
\forall_2 x & 0 & a & 1 \\ 
\end{array}
.   
\]
One can check that $\mathcal{QSSM}^{\mathcal{W}}(X)=\mathcal{MOP}^{\mathcal{S}}(X)=\{(\exists_1,\forall_1), (\exists_2,\forall_2)\}$. 
We can see that $\forall_1((a\ra \forall_1 a)\ra a)=1$, while 
$(\forall_1 a\ra \forall_1 a)\ra \forall_1 a=0$, so that, axiom $(MH_7)$ is not verified for $x=y:=a$. 
\end{proof}

Propositions \ref{qinv-iqb-50-10} and \ref{qinv-60-20}$(2)$ lead to the following result. 

\begin{theorem} \label{qinv-iqb-60} Let $(H,\odot,\ra,0,1)$ be a bounded Wajsberg hoop and let 
$\forall:H\longrightarrow H$ satisfying the axioms $(U_1)-(U_5)$ and 
$(MH_7^{\prime})$ $\forall((x\ra \forall y)\ra x)\le \forall((x\ra \forall y)\ra \forall x)$, for all $x, y\in X$. 
Then $\forall$ is a universal quantifier on $H$. 
\end{theorem}

\begin{remark} \label{qinv-iqb-70}
For an unbounded Wajsberg hoop $(H,\odot,\ra,1)$, the assertions stated by axioms $(U_6)$ 
and $(U_7)$ were proved in \cite{Cim1} based on the definition $(WS)$ of the operation $\oplus$ and 
axiom $(MH_7)$ (see \cite[Lemma 3.2]{Cim1}).
\end{remark}

$\vspace*{5mm}$

\section{Concluding remarks}

We mentioned that the quantum B-algebras provide a unified semantic for non-commutative algebraic logic, 
and all implicational algebras studied before are quantum B-algebras. This motivated us to define and 
investigate in a previous work the monadic quantum B-algebras as generalization of allmost monadic algebras of 
fuzzy structures previously studied.
Since the quantifiers on involutive structures have special algebraic properties, in this paper we defined the 
involutive quantum B-algebras and introduced and studied the existential and universal quantifiers on these 
structures. As the main results, we proved that there is a one-to-one correspondence between the quantifiers on  
weakly involutive quantum B-algebras, and that any pair of quantifiers is a monadic operator on these structures. 
The results proved for the case of involutive pseudo BCK-algebras are also valid for pseudo MV-algebras and 
bounded Wajsberg hoops. \\
In a future work we will define and study the state monadic quantum B-algebras.  
Based on the monadic deductive systems, the uniform topology on a monadic pseudo-equality algebra has been 
introduced and studied in \cite{Ghor1}. 
As another topic of research, one could define a topology on monadic quantum B-algebras.

%\begin{center}
%\sc Acknowledgement 
%\end{center}
%The author is very grateful to the anonimous referees for their useful remarks and suggestions on the subject that %helped improving the presentation.

$\vspace*{5mm}$

\vspace*{3mm}

\end{document}